\documentclass[a4paper,12pt,dvipdfmx]{amsart}
\usepackage[hiresbb]{graphicx}
\usepackage{amsmath,amssymb,amsfonts,amsthm,color}
\newtheorem{definition}{Definition}[section]
\newtheorem{theorem}[definition]{Theorem}
\newtheorem{lemma}[definition]{Lemma}

\newtheorem{remark}[definition]{Remark}
\newtheorem{example}[definition]{Example}
\newtheorem{proposition}[definition]{Proposition}

\newcommand{\C}{\mathbb{C}}

\begin{document}

\title[Non-linear traces on compact operators]{Non-linear traces on the algebra of compact operators and majorization}
%\author{Masaru Nagisa and Yasuo Watatani}

\author{Masaru Nagisa}
\address[Masaru Nagisa]{Department of Mathematics and Informatics, Faculty of Science, Chiba University, 
Chiba, 263-8522,  Japan: \ Department of Mathematical Sciences, Ritsumeikan University, Kusatsu, Shiga, 525-8577,  Japan}
\email{nagisa@math.s.chiba-u.ac.jp}
\author{Yasuo Watatani}
\address[Yasuo Watatani]{Department of Mathematical Sciences,
Kyushu University, Motooka, Fukuoka, 819-0395, Japan}
\email{watatani@math.kyushu-u.ac.jp}

\maketitle

\begin{abstract}
We study non-linear traces of Choquet type and Sugeno type on the algebra of compact operators. 
They have certain partial additivities. We show that these partial additivities characterize 
non-linear traces of both Choquet type and Sugeno type respectively. 
There exists a close relation between non-linear traces of Choquet type and majorization theory. 
We study trace class operators for non-linear traces of Choquet type. More generally we discuss Schatten-von Neumann $p$-class 
operators for non-linear traces of Choquet type. We determine when they form Banach spaces.  This is an attempt  
of non-commutative integration theory for non-linear traces of Choquet type on the algebra of compact operators.  
 We also consider the triangle inequality for  non-linear traces of Sugeno type. 



\end{abstract}

\section{Introduction}
We studied several classes of general non-linear positive maps between $C^*$-algebras 
in \cite{nagisawatatani}.  In \cite{nagisawatatani2} we introduced non-linear traces  
of Choquet type and Sugeno type on matrix algebras. In this paper we study non-linear traces  
of Choquet type and Sugeno type on the algebra $K(H)$ of compact operators.  
We study trace class operators for non-linear traces of Choquet type. More generally 
we discuss  Schatten-von Neumann $p$-class operators for non-linear traces of Choquet type. 
This is an attempt of non-commutative integration theory for non-linear traces of Choquet type 
on the algebra of compact operators.

Ando-Choi \cite{A-C} and Arveson \cite{Ar2} initiated the study of  non-linear completely 
positive maps and extend the Stinespring dilation theorem. 
Hiai-Nakamura \cite{H-N} considered a non-linear counterpart of Arveson's 
Hahn-Banach type extension theorem \cite {Ar1} for completely positive linear maps. 
Bel\c{t}it\u{a}-Neeb \cite{B-N} studied non-linear completely positive maps and dilation 
theorems for real involutive algebras.  
Recently Dadkhah-Moslehian \cite{D-M} investigate some properties of non-linear positive maps 
like Lieb maps and the multiplicative domain for 3-positive maps. Dadkhah-Moslehian-Kian 
\cite{D-M-K} study continuity of non-linear positive maps between $C^*$-algebras.

The functional calculus by a continuous positive function  is a typical example of non-linear positive maps.  
See, for example, \cite{bhatia1} and \cite{bhatia2}. %, \cite{D} and \cite{Si}. 
The functional calculus by operator monotone functions is important to study operator means 
in Kubo-Ando theory in \cite{kuboando}.  
Another important motivation of the study of non-linear positive maps on $C^*$-algebras is
non-additive measure theory, which was initiated by Sugeno \cite{Su} and  Dobrakov \cite{dobrakov}. 
Choquet integrals \cite{Ch} and Sugeno integrals \cite{Su} are studied as non-linear integrals in 
non-additive measure theory. The differences of them are two operations used: 
sum and product for Choquet integrals and maximum and minimum for Sugeno integrals. 
They have partial additivities. 
Choquet integrals have  comonotonic additivity and 
Sugeno integrals have comonotonic F-additivity (fuzzy additivity). Conversely it is known that 
Choquet integrals and Sugeno integrals are characterized by these partial additivities,  see, for example, \cite{De}, 
\cite{schmeidler}, \cite{C-B}, \cite{C-M-R}, \cite{D-G}. 
More precisely, 
comonotonic additivity, positive homogeneity and monotony characterize 
Choquet integrals. Similarly comonotonic F-additivity, F-homogeneity and monotony
characterize Sugeno integrals.
We also note that the inclusion-exclusion integral by Honda-Okazaki \cite{H-O}
for non-additive monotone measures is a general notion of non-linear integrals.

In \cite{nagisawatatani2} we introduced  non-linear traces  
of Choquet type and Sugeno type on matrix algebras and studied a close relation among non-linear traces of 
Choquet type, majorization, unitarily invariant norms and 2-positivity. 
In this paper we study non-linear traces of Choquet type and Sugeno type on the algebra of compact operators,  which are 
operator versions of Choquet integrals and Sugeno integrals.  We 
characterize  non-linear traces of Choquet type and Sugeno type by partial additivities. 
We shall show that non-linear traces of Choquet type have comonotonic additivity and 
non-linear traces of Sugeno type have comonotonic F-additivity (fuzzy additivity). Conversely we can show that 
these partial additivities together with lower semi-continuity in norm,  positive homogeneity, unitary invariance and monotony
characterize  
non-linear traces of both Choquet type and Sugeno type respectively on compact operators $K(H)$.

We also show that there exists a close relation among 
non-linear traces of Choquet type and majorization.  We study trace class operators for non-linear traces of Choquet type.  
We also discuss Schatten-von Neumann $p$-class operators for non-linear traces of Choquet type. 
We determine when they form Banach spaces.  
   This is an attempt of non-commutative integration theory for non-linear traces of Choquet type 
on the algebra of compact operators.  We also consider the triangle inequality for  non-linear traces of Sugeno type.

\vspace{3mm}

This work was supported by JSPS KAKENHI Grant Number \linebreak
JP17K18739.

%%%
%%%      Section 2
%%%

\section{Non-linear traces of Choquet type}
In this section we study non-linear traces of Choquet type on the algebra $K(H)$ of compact operators 
on a separable infinite dimensional Hilbert space $H$.

%%%%%%%%%%%%%%%%%%%%%%%%%%%%%
%     Definition 2.1
%%%%%%%%%%%%%%%%%%%%%%%%%%%%%
\begin{definition} \rm
Let $\Omega$ be a set and ${\mathcal B}$ a ring of sets on $\Omega$, that is, 
${\mathcal B}$ is a family of subsets of $\Omega$ which is closed under the operations 
union $\cup$ and set theoretical difference $\backslash$. Hence  ${\mathcal B}$ is also 
closed under intersection $\cap$. 
A function $\mu: {\mathcal B} \rightarrow [0, \infty]$ is  called a 
{\it monotone measure} if $\mu$ satisfies 
\begin{enumerate}
\item[$(1)$]  $\mu(\emptyset) = 0$, and 
\item[$(2)$] For any $A,B \in {\mathcal B}$, if $A \subset B$, 
then $\mu(A) \leq \mu(B)$. 

\end{enumerate}
\end{definition}

We recall  the discrete Choquet integral with respect to a monotone measure on a finite set 
$\Omega  = \{1,2, \dots, n\}$.  Let ${\mathcal B} = P(\Omega)$ be the set of all 
subsets of $\Omega$ and $\mu:  {\mathcal B} \rightarrow [0, \infty)$ be a finite 
monotone measure, where `finite' means that $\mu(\Omega) < \infty$. 

%%%%%%%%%%%%%%%%%%%%%%%%%%%%%
%     Definition 2.2
%%%%%%%%%%%%%%%%%%%%%%%%%%%%%
\begin{definition} \rm 
The  discrete Choquet integral  of $f = (x_1, x_2,\dots, x_n) \in [0,\infty)^n$ 
with respect to a monotone measure $\mu$ on a finite set 
$\Omega  = \{1,2, \dots, n\}$ is defined as follows:  
$$
{\rm (C)}\int f d\mu = \sum_{i=1}^{n-1} (x_{\sigma(i) }- x_{\sigma(i+1)})\mu(A_i) 
+ x_{\sigma(n) }\mu(A_n) , 
$$
where $\sigma$ is a permutation on $\Omega$ such that  
$x_{\sigma(1)} \geq x_{\sigma(2)} \geq  \dots \geq x_{\sigma(n)}$ 
and  $A_i = \{\sigma(1),\sigma(2),\dots,\sigma(i)\}$.  
Here we should note that 
$$
f = \sum_{i=1}^{n-1} (x_{\sigma(i) }- x_{\sigma(i+1)})\chi_{A_i} 
+ x_{\sigma(n) }\chi_{A_n} .
$$
\end{definition}

Let $A = {\mathbb C}^n$ and define 
$({\rm C-}\varphi)_{\mu} :( {\mathbb C}^n)^+ \rightarrow {\mathbb C}^+$ 
by the  Choquet integral $({\rm C-}\varphi)_{\mu}(f) = {\rm (C)}\int f d\mu$.  Then 
$({\rm C-}\varphi)_{\mu}$ is a non-linear positive map and 
it is {\it monotone} in the sense that, if $0 \leq f \leq g$ then 
${\rm (C)}\int f d\mu \leq {\rm (C)}\int g d\mu$. It is {\it positively homogeneous} in the sense that 
${\rm (C)}\int kf d\mu = k  \ {\rm (C)}\int f d\mu$ for a scalar $k \geq 0$.

Real valued functions $f$ and $g$ on a set $\Omega$ are said to be {\it comonotonic} if 
$(f(s) -f(t))(g(s)-g(t)) \geq 0$ for any $s,t \in \Omega$, that is, $f(s) < f(t)$ implies $g(s) \leq g(t)$ 
for any $s,t \in \Omega$. In particular consider when  $\Omega = \{x_1, x_2, \dots,x_n\}$ is a finite $n$ points set. 
Then $f$ and $g$ are comonotonic if and only if there exists a permutation $\sigma$ on $\{1,2,\dots,n\}$ such that 
\begin{gather*}
f(x_{\sigma(1)}) \geq f(x_{\sigma(2)}) \geq \dots \geq  f(x_{\sigma(n)})  \\
 \text{and }
g(x_{\sigma(1)}) \geq g(x_{\sigma(2)}) \geq \dots \geq  g(x_{\sigma(n)}).
\end{gather*}
We should note that comonotonic relation is not transitive and not an equivalent relation. 
The Choquet integral has {\it comonotonic additivity}, that is, 
for any comonotonic pair $f$ and $g$, 
$$
{\rm (C)}\int (f  + g )d\mu= {\rm (C)}\int f d\mu  + {\rm (C)}\int g d\mu . 
$$
The Choquet integral has comonotonic additivity, positive homogeneity and monotony. Conversely 
 comonotonic additivity, positive homogeneity and monotony  characterize the Choquet integral. 

We shall consider an operator version of the discrete Choquet integral and its characterization. 
We introduce non-linear traces of Choquet type on the algebra $K(H)$ of compact operators 
on a separable infinite dimensional Hilbert space $H$.  We denote by $B(H)$ the set of bounded linear operators and 
by $F(H)$ the set of finite rank operators.
Let $K(H)^+$ (respectively $F(H)^+$) be the set of compact operators (respectively the set of finite rank operators) 
which are positive. For a topological space $X$, let $C(X)$ be the set of all continuous functions on $X$ and 
$C(X)^+$ be the set of all non-negative continuous functions on $X$. 

%%%%%%%%%%%%%%%%%%%%%
%% Definition 2.3
%%%%%%%%%%%%%%%%%%%%%
\begin{definition} \rm A non-linear positive map 
$\varphi : K(H)^+ \rightarrow  [0, \infty]$ is called a {\it trace} if 
$\varphi$ is unitarily invariant, that is, 
$\varphi(uau^*) = \varphi(a)$ for any $a \in K(H)^+$ and any unitary 
$u \in B(H)$.  
\begin{itemize}
  \item  $\varphi$ is {\it monotone} if 
$a \leq b$ implies  $\varphi(a) \leq  \varphi(b)$ for any $a,b \in  K(H)^+$. 
  \item  $\varphi$ is {\it positively homogeneous} if 
$\varphi(ka) = k\varphi(a)$ for any $a \in K(H)^+$ and any  scalar $k \geq 0$, 
where we regard $0 \cdot \infty = 0$. 
  \item $\varphi$ is {\it comonotonic additive on the spectrum} if 
$$
\varphi(f(a)  + g(a)) = \varphi(f(a)) + \varphi(g(a)) 
$$  
for any $a \in K(H)^+$ and 
any comonotonic functions $f$ and $g$  in $C(\sigma(a))^+$ with $f(0) = g(0) = 0$, where 
$f(a)$ is a functional calculus of $a$ by $f$ and  $C(\sigma(a))$ is 
the set of continuous functions on the spectrum $\sigma(a)$ of $a$. 
  \item $\varphi$ is {\it monotonic increasing additive on the spectrum} if 
$$
\varphi(f(a)  + g(a)) = \varphi(f(a)) + \varphi(g(a))
$$  
for any $a \in K(H)^+$ and 
any monotone increasing functions $f$ and $g$  in $C(\sigma(a))^+$ with $f(0)=g(0)=0$. 
Then by induction, 
we also have 
$$
 \varphi(\sum_{i=1}^n f_i(a)) = \sum_{i=1}^n \varphi (f_i(a))
% \varphi(f_1)(a) + f_2(a) + \dots +f_n(a)) = \varphi(f_1(a)) + \varphi(f_2(a)) + \dots +\varphi(f_n(a))
$$
for any monotone increasing functions $f_1, f_2, \dots, f_n$ in  $C(\sigma(a))^+$  with  $f_i(0) = 0$ for $i=1,2,\ldots, n$.
\end{itemize}
\end{definition}

Let ${\mathbb N}$ be the set of positive integer and put ${\mathbb N}_0 = {\mathbb N} \cup \{0\}$. 
Let  ${\mathcal R}({\mathbb N})$ be the ring of all finite subsets of ${\mathbb N}$.  
A monotone measure $\mu: {\mathcal R}({\mathbb N}) \rightarrow [0, \infty]$ on  ${\mathbb N}$ is {\it permutation invariant} if 
 $\mu(A) = \mu(\sigma (A))$ for any finite subset $A \subset {\mathbb N}$ and any permutation $\sigma$ on 
 ${\mathbb N}$. 
The following lemma is clear.

%%%%%%%%%%%%%%%%%%%%%%%%%%%%%
%     Lemma 2.4
%%%%%%%%%%%%%%%%%%%%%%%%%%%%%
\begin{lemma} Let $\alpha: {\mathbb N}_0  \rightarrow [0, \infty)$ be a 
monotone increasing function with $\alpha(0) = 0$, that is, 
$$
0 = \alpha(0) \leq  \alpha(1) \leq  \alpha(2) \leq \cdots \leq  \alpha(n) \leq \cdots .
$$
Define a monotone measure $\mu_{\alpha}: {\mathcal R}({\mathbb N}) \rightarrow [0, \infty]$ on 
${\mathbb N}$ by 
$$
\mu_{\alpha}(A) = \alpha( ^{\#}A) \ \ \ \text{for a finite subset } A \subset {\mathbb N} ,
$$
where $^{\#}A$ is the cardinality of $A$. 

Then $\mu_{\alpha}$ is a permutation invariant 
monotone measure such that  $\mu_{\alpha}(A)$ is finite for any finite subset  $A$ of  $\mathbb{N}$. 
Conversely any  permutation invariant monotone measure 
$\mu : {\mathcal R}({\mathbb N}) \rightarrow [0, \infty]$ on ${\mathbb N}$ such that 
$\mu(A)$ is finite for any finite subset $A$ of  $\mathbb{N}$ has this form. 
\end{lemma}

%%%%%%%%%%%%%%%%%%%%%%%%%%%%%%%
%     Definition 2.5
%%%%%%%%%%%%%%%%%%%%%%%%%%%%%%%
\begin{definition} \rm
Let $\alpha: {\mathbb N}_0  \rightarrow [0, \infty)$ be a 
monotone increasing function with $\alpha(0) = 0$, that is, 
$$
0 = \alpha(0) \leq  \alpha(1) \leq  \alpha(2) \leq \cdots \leq  \alpha(n) \leq \cdots .
$$
We denote by 
$\mu_{\alpha}$ the associated permutation invariant 
monotone measure on ${\mathbb N}$. 
Define 
$\varphi_{\alpha} : K(H)^+ \rightarrow  [0, \infty]$
as follows:  Any $a \in K(H)^+$ can be written as 
$a = \sum_{n=1}^{\infty} \lambda_n(a)p_n$, where each $p_n$ is a one dimensional spectral projection of $a$ 
with  $\sum_{n=1}^{\infty} p_n \leq I$, and 
$$
\lambda(a) := (\lambda_1(a),\lambda_2(a),\ldots, \lambda_n(a), \ldots)
$$
is the list of the 
eigenvalues of $a$ in decreasing order :
$\lambda_1(a) \geq \lambda_2(a) \geq \cdots \geq \lambda_n(a) \geq \cdots $ with 
counting multiplicities, where  the sequence $(\lambda_n(a))_n$ converges to zero.  
In order to make the sequence uniquely determined, we use the following convention : 
If $a \in  K(H)^+$ is not a finite-rank operator, then each $\lambda_n(a) \not= 0$. 
If $a \in  K(H)^+$ is a finite-rank operator, $\lambda(a)$ is given by adding zeroes to the list of eigenvalues of $a$. 
In both case, the spectrum $\sigma(a)$ of $a$ is given by 
$\sigma(a) = \cup_{n = 1}^{\infty} \{\lambda_n(a)\}   \cup \{0\}$.  We call such 
$\lambda(a)$ the {\it eigenvalue sequence} of $a \in  K(H)^+$. 

Let
\begin{align*}
\varphi_{\alpha}(a) &= \sum_{i=1} ^{\infty} ( \lambda_i(a)- \lambda_{i+1}(a))\mu_{\alpha}(A_i) \\
&= \sum_{i=1} ^{\infty} ( \lambda_i(a)- \lambda_{i+1}(a)) \alpha( ^{\#}A_i) \\
& =  \sum_{i=1} ^{\infty} ( \lambda_i(a)- \lambda_{i+1}(a)) \alpha(i) ,
\end{align*}
where $A_i = \{1,2,\dots,i \}$.
We call $\varphi_{\alpha}$ the non-linear trace of Choquet type associated with $\alpha$.  
Note that $\varphi_{\alpha}$ is lower semi-continuous on $K(H)^+$ in norm in the following sense:
When $\varphi_{\alpha}(a) < \infty$, for any $\epsilon > 0$ there exists $\delta > 0$ such that for any 
$x \in  K(H)^+$, if $\| x - a\| < \delta$, then $\varphi_{\alpha}(a) - \epsilon < \varphi_{\alpha}(x).$  
When $\varphi_{\alpha}(a) = \infty$,  for any $M > 0$ there exists $\delta > 0$ such that for any 
$x \in  K(H)^+$, if $\| x - a\| < \delta$, then $M< \varphi_{\alpha}(x).$  
In fact, let 
$$
S_n(a) = \sum_{i=1} ^{n} ( \lambda_i(a)- \lambda_{i+1}(a)) \alpha(i).
$$
Since 
each $\lambda_i$ is norm continuous on $K(H)^+$,  $S_n: K(H)^+ \rightarrow  [0, \infty)$ 
is also norm continuous on  $K(H)^+$.  Moreover 
$$
\varphi_{\alpha}(a)  = \sup \{S_n(a) \ | \ n \in  {\mathbb N} \}.
$$
Therefore $\varphi_{\alpha}$ is lower semi-continuous on $K(H)^+$ in norm. 
\end{definition}

%%%%%%%%%%%%%%%%%%%%%%%%%%%%%%%%
%     Example 2.6
%%%%%%%%%%%%%%%%%%%%%%%%%%%%%%%%
\begin{example} \rm  
We sometimes denote $\alpha$ by $\alpha = (\alpha(0), \alpha(1),\alpha(2), \dots)$. 
\begin{enumerate}
\item[(1)] If $\alpha =(0,1,2,3,\dots)$, then 
$\varphi_{\alpha}(a) = {\rm Tr}(a) = \sum_{i=1}^{\infty} {\ \lambda_i(a)}$ is the usual linear trace. 
\item[(2)] If $\alpha =(0,1,1,1,\dots)$, then $\varphi_{\alpha}(a) = \lambda_1(a)$.
\item[(3)] If $\alpha =(\overbrace{0,\ldots,0}^i, 1,1,\dots)$ $(i=2,3,\ldots)$, 
then $\varphi_{\alpha}(a) = \lambda_i(a)$.
\item[(4)] If $\alpha = (0,1,2,\dots,k,k,\dots)$, then $\varphi_{\alpha}(a) =  \sum_{i=1} ^{k}  \lambda_i(a)$. 
\end{enumerate}
\end{example}

We can describe non-linear traces of Choquet type in another way. 
%%%%%%%%%%%%%%%%%%%%%%%%%%%%%
%    Proposition 2.7
%%%%%%%%%%%%%%%%%%%%%%%%%%%%%
\begin{proposition} 
Let $\alpha: {\mathbb N}_0  \rightarrow [0, \infty)$ be a 
monotone increasing function with $\alpha(0) = 0$. Put 
$c_n = \alpha(n) - \alpha(n-1) \geq 0$ and $c_1 = \alpha(1) - \alpha(0) = \alpha(1)$. 
Then $\alpha(n) = \sum_{i=1}^n c_i$ and 
the non-linear trace $\varphi_{\alpha}$ of Choquet type associated with $\alpha$ is also 
described as follows: For  $a \in K(H)^+$, 
 $$
 \varphi_{\alpha}(a) = \sum_{i=1} ^{\infty}  \lambda_i(a) (\alpha(i) - \alpha(i-1)) 
 = \sum_{i=1} ^{\infty}  \lambda_i(a) c_i .
 $$
 Conversely for any sequence $c =(c_n)_n$ with each $c_i \in [0,\infty)$, the non-linear trace $\varphi$ 
 defined by 
 $$
  \varphi(a) = \sum_{i=1} ^{\infty}  \lambda_i(a) c_i
 $$
 is a trace of Choquet type associated with some $\alpha$ such that  $ \alpha(n) = \sum_{i=1}^n c_i$. 
\end{proposition}
\begin{proof}
For $a \in K(H)^+$, 
put $r_i := \lambda_i(a) - \lambda_{i+1}(a) \geq 0$. Then  we have that $ \sum_{i=n} ^{\infty} r_i  = \lambda_n(a)$. 
Let us consider the double series $\sum \sum_{(i,j) \in \Lambda } r_i c_j $
with non-negative terms, 
%Let us consider the following positive term double sums: 
%$\sum \sum_{(i,j) \in \Lambda } r_i c_j $,
where index $(i,j)$ runs in the lower triangular region
$\Lambda = \{(i,j) \in {\mathbb N}^2 | i \leq j\}$. 
Since the value of this double series
%Since the value of the positive term double sums  
does not depend on the order of taking sums, we have that 
$$
\sum \sum_{(i,j) \in \Lambda } r_i c_j  = \sum_{i=1} ^{\infty}\sum_{j=1} ^{i }  r_i c_j 
= \sum_{j=1} ^{\infty}\sum_{i=j} ^{\infty }  r_i c_j 
$$
That is, $\sum_{i=1} ^{\infty} r_i\alpha(i) =  \sum_{j=1} ^{\infty} \lambda_j(a)c_j$.
Therefore we have that 
$$
 \varphi_{\alpha}(a)  =  \sum_{i=1} ^{\infty} ( \lambda_i(a)- \lambda_{i+1}(a)) \alpha(i) 
 = \sum_{i=1} ^{\infty}  \lambda_i(a) c_i. 
$$
\end{proof}

We shall characterize non-linear traces of Choquet type. 

%%%%%%%%%%%%%%%%%%%%%%%%%%%%%
%     Theorem 2.8
%%%%%%%%%%%%%%%%%%%%%%%%%%%%%
\begin{theorem} 
Let $\varphi : K(H)^+ \rightarrow  [0, \infty]$ be a non-linear 
positive map.  Then the following conditions are equivalent:
\begin{enumerate}
\item[$(1)$]  $\varphi$ is a non-linear trace $\varphi = \varphi_{\alpha}$  of Choquet type associated with  
a monotone increasing function $\alpha: {\mathbb N}_0 \rightarrow [0, \infty)$  with $\alpha(0) = 0$.
\item[$(2)$] $\varphi $ is monotonic increasing additive on the spectrum, unitarily invariant, monotone, 
positively homogeneous and lower semi-continuous on $K(H)^+$ in norm.  Moreover $\varphi(a) < \infty  $ 
for any finite rank operator $a \in F(H)^+$. 
\item[$(3)$] $\varphi $ is comonotonic additive on the spectrum, unitarily invariant, monotone, 
positively homogeneous and lower semi-continuous on $K(H)^+$ in norm. Moreover $\varphi(a) < \infty  $ 
for any finite rank operator $a \in F(H)^+$. 
\end{enumerate}
\end{theorem}
\begin{proof}
(1)$\Rightarrow$(2)  Assume that $\varphi$ is a non-linear trace $\varphi = \varphi_{\alpha}$  of Choquet type associated with  
a monotone increasing function $\alpha$. 
For $a \in K(H)^+$, 
let 
$\lambda(a) = (\lambda_1(a),\lambda_2(a),\ldots, \lambda_n(a), \ldots)$ be the list of the 
eigenvalues of $a$ in decreasing order :
$\lambda_1(a) \geq \lambda_2(a) \geq \cdots \geq \lambda_n(a) \geq \cdots $ with 
counting multiplicities, where  the sequence $(\lambda_n(a))_n$ converges to zero.  We have  
$a = \sum_{n=1}^{\infty} \lambda_n(a)p_n$, where each $p_n$ is a one dimensional spectral projection of $a$. 
Then the spectrum $\sigma(a) = \{\lambda_i(a) \ | \ i = 1,2, \dots \} \cup \{0\}$. 
For any monotone increasing functions $f$ and $g$  in $C(\sigma(a))^+$ with $f(0) = g(0) = 0$, 
we have the  spectral decompositions 
\begin{gather*}
f(a) = \sum_{i=1}^{\infty} f(\lambda_i(a)) p_i, \ \ \ g(a) = \sum_{i=1}^{\infty} g(\lambda_i(a)) p_i,   \\
\text{and } (f+g)(a) = \sum_{i=1}^{\infty}(f(\lambda_i(a))+ (g(\lambda_i(a))) p_i. 
\end{gather*}
Since $f+g$ is also monotone increasing on $\sigma(a)$ as well as $f$ and $g$, we have 
\begin{gather*}
f(\lambda_1(a)) \geq f(\lambda_2(a)) \geq \cdots \geq f(\lambda_n(a)) \geq \cdots,    \\
g(\lambda_1(a)) \geq g(\lambda_2(a)) \geq \cdots \geq g(\lambda_n(a))\geq \cdots,   \\
\text{and }  (f+g)(\lambda_1(a)) \geq (f+g)(\lambda_2(a)) \geq \cdots \geq (f+g)(\lambda_n(a))\geq \cdots.
\end{gather*}
Since $f$ and $g$  in $C(\sigma(a))^+$ with $f(0) = g(0) = 0$,  the sequences  $(f(\lambda_n(a)))_n$,  
$(g(\lambda_n(a)))_n$ and $((f+g)(\lambda_n(a)))_n$ converge to $0$. Hence  $f(a), g(a)$  and 
$f(a) + g(a) = (f+g)(a)$ are positive compact operators.   
Therefore
\begin{gather*}
\varphi_{\alpha}(f(a)) = \sum_{i=1} ^{\infty} ( f(\lambda_i(a))- f(\lambda_{i+1}(a))) \alpha(i),  \\
\varphi_{\alpha}(g(a)) = \sum_{i=1} ^{\infty} ( g(\lambda_i(a))- g(\lambda_{i+1}(a))) \alpha(i) , 
\end{gather*}
and 
\begin{align*}
  & \varphi_{\alpha}(f(a)+g(a))  = \varphi_{\alpha}((f+g)(a)) \\
 = & \sum_{i=1} ^{\infty} ( (f+g)(\lambda_i(a))- (f+g)(\lambda_{i+1}(a))) \alpha(i) \\
 = & \varphi_{\alpha}(f(a))  + \varphi_{\alpha}(g(a)) . 
\end{align*}
Thus $\varphi_{\alpha}$ is monotonic increasing additive on the spectrum. 

For $a,b \in K(H)^+$, let 
$\lambda(a) = (\lambda_1(a),\lambda_2(a),\dots, \lambda_n(a), \dots)$ and 
$\lambda(b) = (\lambda_1(b),\lambda_2(b),\dots,$ $\lambda_n(b), \dots )$ be  the 
eigenvalue sequences of $a$ and $b$. 
Assume that $a \leq b$. 
By the mini-max  principle for eigenvalues, we have that 
$\lambda_i(a) \leq \lambda_i(b)$ for $i = 1,2,3,\dots$.  Hence 
$$
\varphi_{\alpha}(a) 
= \sum_{i=1} ^{\infty}  \lambda_i(a)c_i 
\leq \sum_{i=1} ^{\infty}  \lambda_i(b)c_i 
= \varphi_{\alpha}(b). 
$$
Thus $\varphi_{\alpha}$ is monotone. 

For a positive scalar $k$,  $ka = \sum_{i=1}^{\infty} k \lambda_i(a) p_i$ and 
$$
\lambda(ka) = (k\lambda_1(a),k\lambda_2(a),\dots,k \lambda_n(a), \dots )
$$ 
is the list of the 
eigenvalues of $ka$ in decreasing order, 
$$
\varphi_{\alpha}(ka) = \sum_{i=1} ^{\infty} ( k \lambda_i(a)- k \lambda_{i+1}(a)) \alpha(i) 
 = k \varphi_{\alpha}(a). 
$$
Thus $\varphi_{\alpha}$ is positively homogeneous. 
 It is clear 
that $\varphi_{\alpha}$ is unitarily invariant by definition. We already showed that $\varphi_{\alpha}$ is 
lower semi-continuous on $K(H)^+$ in norm. Moreover it is trivial that $\varphi(a) < \infty  $ 
for any finite rank operator $a \in F(H)^+$ by definition. 

(2)$\Rightarrow$(1)  Assume that $\varphi $ is monotonic increasing additive on the spectrum, unitarily invariant, monotone and 
positively homogeneous. We also assume that $\varphi $ is lower semi-continuous on $K(H)^+$ in norm and 
 $\varphi(a) < \infty  $ for any finite rank operator $a \in F(H)^+$.
Let $I = \sum_{n=1}^{\infty} p_n$ be the decomposition of the identity by 
minimal projections. Define  a function $\alpha: {\mathbb N}_0 \rightarrow [0, \infty)$  with $\alpha(0) = 0$
by $\alpha (i) = \varphi(p_1+ p_2 + \dots + p_i)$.  Since  $\varphi $ is unitarily invariant, $\alpha$ does not depend 
on the choice of minimal projections. Since  $\varphi $ is monotone, $\alpha$ is monotone increasing. 
For a positive finite rank operator $b \in F(H)^+$, 
let 
$b = \sum_{i=1}^{n} \lambda_i(b) q_i$
be the spectral decomposition of $b$, where 
$\lambda(b) = (\lambda_1(b),\lambda_2(b),\dots, \lambda_n(b), 0, \dots)$ is the 
eigenvalue sequence of $b$

Now define functions  $f_1,f_2,\dots, f_n \in C(\sigma(b))^+$  by 
$$
f_i (x) = \begin{cases}    \lambda_i(b) - \lambda_{i+1}(b), \quad & \text{ if } x\in  \{ \lambda_1(b), \lambda_2(b)\dots, \lambda_i(b) \}  \\
                                 0  &  \text{ if }  x\in \{ \lambda_{i+1}(b), \dots, \lambda_n(b), 0 \}
\end{cases}
$$
for $i = 1,2,\dots, n-1$ and
$$
f_n (x) = \begin{cases}    \lambda_n(b)  & \text{ if } x\in  \{ \lambda_1(b), \lambda_2(b)\dots, \lambda_n(b) \}  \\
                                 0  &  \text{ if }  x = 0 . 
\end{cases}
$$
Then  $f_1, f_2, \dots, f_n$ are monotone increasing functions in  $C(\sigma(b))^+$ with $f_i(0) = 0$ for $i=1,\dots,n$ 
such that 
$$
f_1(x) + f_2(x) + \dots +f_n(x) = x  \ \ \ \text{ for } x \in \{ \lambda_1(b),\lambda_2(b),\dots, \lambda_n(b), 0 \} .
$$
Therefore 
$$
f_1(b) + f_2(b) + \dots +f_n(b) = b .
$$
Moreover we have that 
$$
f_i(b) =  (\lambda_i(b) - \lambda_{i+1}(b))(q_1 + q_2 + \dots +q_i) \quad \text{ for } i=1,2,\dots, n-1
$$
and $f_n (b) =  \lambda_n(b) (q_1 + q_2 + \dots +q_n).$

Since $\varphi $ is monotonic increasing additive on the spectrum and 
positively homogeneous, we have that 
\begin{align*}
   & \varphi(b)  = \varphi(f_1(b) + f_2(b) + \dots +f_n(b)) \\
  =&  \varphi(f_1(b)) + \varphi(f_2(b)) + \dots +\varphi(f_n(b)) \\
  =&  \sum_{i=1} ^{n-1} \varphi( ( \lambda_i(b)- \lambda_{i+1}(b))(q_1 + q_2 + \dots q_i)) 
+ \varphi(\lambda_n(b) (q_1 + q_2 + \dots +q_n)\\
  =&  \sum_{i=1} ^{n-1}  ( \lambda_i(b)- \lambda_{i+1}(b))\varphi(q_1 + q_2 + \dots q_i) 
+ \lambda_n(b)\varphi (q_1 + q_2 + \dots +q_n)\\
  =&  \sum_{i=1} ^{n-1}  ( \lambda_i(b)- \lambda_{i+1}(b))\alpha(i) 
+ \lambda_n(b)\alpha(n) = \varphi_{\alpha}(b).
\end{align*}
Next, 
for any positive compact operator $a \in K(H)^+$,  let 
$\lambda(a) = (\lambda_1(a),\lambda_2(a),\dots, \lambda_n(a), \dots)$ be the 
eigenvalue sequence of $a$.  
We have  
$a = \sum_{n=1}^{\infty} \lambda_n(a)p_n$, where each $p_n$ is a one dimensional spectral projection of $a$.  
Put $b_n = \sum_{k=1}^{n} \lambda_k(a)p_k$.  Then the sequence $(b_n)_n$ is increasing and converges 
to $a$ in norm.  Since both $\varphi$ and $\varphi_{\alpha}$  are lower semi-continuous on $K(H)^+$ in norm, monotone maps 
and 
$\varphi(b_n) = \varphi_{\alpha}(b_n)$, we conclude that  $ \varphi(a) = \varphi_{\alpha}(a)$ by 
taking their limits. 
Therefore $\varphi$ is equal to the  non-linear trace $\varphi_{\alpha}$  of Choquet type associated with  
a monotone increasing function $\alpha$. 

(3)$\Rightarrow$(2)  It is clear from the fact that any monotone increasing functions $f$ and $g$  in $C(\sigma(a))$ 
are comonotonic.

(1)$\Rightarrow$(3)  For $a \in K(H)^+$, 
let $a = \sum_{i=1}^{\infty} \lambda_i(a) p_i$
be the spectral decomposition of $a$, where 
$\lambda(a) = (\lambda_1(a),\lambda_2(a),\dots, \lambda_n(a),\dots )$ is the eigenvalue sequence of $a$. 
Since $a$ is a positive compact operator, for any $\varepsilon > 0$,  
$\{n \in {\mathbb N} | \lambda_n(a) > \varepsilon \}$  
is a finite set. Therefore by induction, we can show  that 
for any comonotonic functions $f$ and $g$  in $C(\sigma(a))$ with $f(0) = g(0) = 0$, there exists an injective map 
$\tau: {\mathbb N} \rightarrow {\mathbb N}$ such that 
\begin{gather*}
f(\lambda_{\tau(1)}(a)) \geq f(\lambda_{\tau(2)}(a)) \geq \dots \geq f(\lambda_{\tau(n)}(a)), \dots  \\ 
g(\lambda_{\tau(1)}(a)) \geq g(\lambda_{\tau(2)}(a)) \geq \dots \geq g(\lambda_{\tau(n)}(a)), \dots \\
\intertext{and} 
  f(a) = \sum_{i=1}^{\infty} f(\lambda_{\tau(i)}(a)) p_{\tau(i)}, \quad
  g(a) = \sum_{i=1}^{\infty} g(\lambda_{\tau(i)}(a)) p_{\tau(i)}.
\end{gather*}
Moreover $\lambda_{\tau(i)}(a)$ converges to $0$. 

For example, consider the case that $f(\lambda_n) >0$ and $g(\lambda_n) >0$ for any $n \in {\mathbb N}$. 
Put $ {\mathbb N}_1 = {\mathbb N}$.  Let 
$M_f ({\mathbb N}_1) = \max \{f(\lambda_n) | n \in {\mathbb N}_1 \}$ and  
$M_g ({\mathbb N}_1) = \max \{g(\lambda_n) | n \in {\mathbb N}_1\}$.
Define 
\begin{gather*}
S_f(1): = \{n \in {\mathbb N} | f(\lambda_n) = M_f ({\mathbb N}_1) \} \not= \emptyset \\
  \text{ and  }  S_g(1): = \{n \in {\mathbb N} | f(\lambda_n) = M_g ({\mathbb N}_1) \} \not= \emptyset.
\end{gather*}  
Then there exists a natural number $n_1 \in S_f(1) \cap  S_g(1)$.  In fact, on the contrary suppose that 
$S_f(1) \cap  S_g(1) = \emptyset$. Choose $i \in  S_f(1)$ and $j \in   S_g(1)$. Then 
$ f(\lambda_i) > f(\lambda_j)$ and $ g(\lambda_i) < g(\lambda_j)$.  This contradicts to that $f$ and $g$ are 
comonotonic.  We define $\tau(1) = n_1$. 

Inductively put ${\mathbb N}_{k+1} = {\mathbb N}\backslash \{n_1,n_2, ..., n_k\}$. 
Let $M_f ({\mathbb N}_{k+1}) =\max $ $\{ f(\lambda_n) | n \in {\mathbb N}_{k+1} \}$ and 
$M_g ({\mathbb N}_{k+1}) = \max \{g(\lambda_n) | n \in { \mathbb N}_{k+1}\}$. 
Define 
\begin{gather*}
S_f(k + 1): = \{n \in {\mathbb N}_{k+1} | f(\lambda_n) = M_f ({\mathbb N}_{k+1}) \} \not= \emptyset  \\
 \text{ and }  S_g(k + 1): = \{n \in {\mathbb N}_{k+1} | f(\lambda_n) = M_g ({\mathbb N}_{k+1}) \} \not= \emptyset. 
\end{gather*}
Then there exists a natural number $n_{k+1} \in S_f(k + 1) \cap  S_g(k + 1)$. We define $\tau(k + 1) = n_{k + 1}$. 
In this way we can construct $\tau: {\mathbb N} \rightarrow {\mathbb N}$, which is a desired injective map. 

In the other cases  that $f(\lambda_n) = 0$ or  $g(\lambda_m) = 0$ for some $n, m \in {\mathbb N}$, 
 $\tau$ can be  similarly constructed  by carefully avoiding such $n$ and $m$. 

Considering this fact, we can prove  (1)$\Rightarrow$(3) similarly as (1)$\Rightarrow$(2). 
\end{proof}

%%%%%%%%%%%%%%%%%%%%%%%%%%%%%%%
%  Remark 2.9
%%%%%%%%%%%%%%%%%%%%%%%%%%%%%%%
\begin{remark} \rm 
As in the above proof, 
for any positive compact operator $a \in K(H)^+$, 
the value $\varphi_{\alpha}(a)$ can be approximated by the value of the canonical finite rank operators $b_n$ 
as follows: 
let $\lambda(a) = (\lambda_1(a),\lambda_2(a),\cdots,$ $\lambda_n(a), \cdots)$ be the 
eigenvalue sequence of $a$.  
We have  
$a = \sum_{n=1}^{\infty} \lambda_n(a)p_n$, where each $p_n$ is a one dimensional spectral projection of $a$.  
Put $b_n = \sum_{k=1}^{n} \lambda_k(a)p_k$.  Then the sequence $(b_n)_n$ of finite rank operators 
is increasing and converges 
to $a$ in norm.  Since  $\varphi_{\alpha}$ is  lower semi-continuous on $K(H)^+$ in norm and monotone map, 
$\varphi_{\alpha}(a) = \lim_{n \to \infty} \varphi_{\alpha}(b_n) $ . 
\end{remark} 

%%%
%%%   Section 3
%%%

\section{Majorization and trace class operators for non-linear traces of Choquet type,}

In this section we discuss relations between  the majorization theory 
for eigenvalues and singular values of compact operators and trace class operators for non-linear traces of Choquet type. 
For  $a \in K(H)^+$, 
let $a = \sum_{i=1}^{\infty} \lambda_i(a) p_i$
be the spectral decomposition of $a$, where 
$\lambda(a) = (\lambda_1(a),\lambda_2(a),\dots, \lambda_n(a),\dots )$ is the list of the 
eigenvalues of $a$ in decreasing order with counting multiplicities. 
For fixed $i = 1,2,\dots $,  we denote by $\lambda_i$ a non-linear map 
$ \lambda_i: (K(H))^+ \rightarrow  {\mathbb C}^+$ given by 
$\lambda_i(a)$ for $a \in (K(H))^+$.

Let $x = (x_n)_n$ be an infinite sequence of non-negative numbers with $x_n \rightarrow 0$ as $n \rightarrow 
\infty.$  The decreasing rearrangement of $x$ is 
denoted by 
$x^{\downarrow} = (x^{\downarrow}_1, x^{\downarrow}_2, \dots, x^{\downarrow}_n, \dots )$. 
For such $x,y$ the weak majorization $x \prec _w  y$ is defined by 
$$
\sum_{i=1}^k x^{\downarrow}_i \leq \sum_{i=1}^k y^{\downarrow}_i \quad ( \text{ for }  k = 1,2,3,\dots).
$$
The majorization $x \prec y$ is defined by 
$$
\sum_{i=1}^k x^{\downarrow}_i \leq \sum_{i=1}^k y^{\downarrow}_i \quad  ( \text{ for }  k = 1,2,3,\dots )
\text{ and } 
\sum_{i=1}^{\infty}x^{\downarrow}_i = \sum_{i=1}^{\infty} y^{\downarrow}_i  .
$$
For $a,b \in (K(H))^+$, 
if $a \leq b$, then $\lambda_i(a) \leq \lambda_i(b)$ for any $i = 1, 2, \dots $. 
If $\lambda_i(a) \leq \lambda_i(b)$ for any $i = 1, 2, 3,\dots$, then $\lambda(a) \prec_w \lambda(b)$.  
If $\lambda(a) \prec  \lambda(b)$,  then $\lambda(a) \prec _w \lambda(b)$.  See, for example,  \cite{bhatia1}, 
\cite{H}, \cite{hiaipetz}  and \cite{simon} for majorization theory of matrices and compact operators.

Moreover  we see that $a \leq b$ if and only if $\varphi(a) \leq \varphi(b)$ for any positive linear 
functional.  
 We shall consider a similar fact for the condition that 
$\lambda_i(a) \leq \lambda_i(b)$ for any $i = 1, 2,3, \dots $. 
The following proposition is just a reformulation of a known fact:

%%%%%%%%%%%%%%%%%%%%%%%%%%%%%%%%
%    Proposition 3.1
%%%%%%%%%%%%%%%%%%%%%%%%%%%%%%%%
\begin{proposition} 
For $a,b \in (K(H))^+$, the following conditions are equivalent:
\begin{enumerate}
\item[$(1)$]  $\varphi_{\alpha}(a) \leq \varphi_{\alpha}(b)$
for any non-linear trace $\varphi_{\alpha}$  of Choquet type associated with  
all monotone increasing function $\alpha: {\mathbb N}_0  \rightarrow [0, \infty)$  with $\alpha(0) = 0$.  
\item[$(2)$] $\lambda_i(a) \leq \lambda_i(b)$ for any $i = 1,2,3,\dots $. 
\item[$(3)$] there exists a contraction $c \in B(H)$ such that $a = cbc^*$.
\end{enumerate}
\end{proposition}
\begin{proof}
(1)$\Rightarrow$(2)  It is clear, because there exists 
$\alpha = (0,\dots,0,1,1,\dots,)$ with $\varphi_{\alpha}(a) = \lambda_i(a) $. 

(2)$\Rightarrow$(1)  Since any  non-linear trace $\varphi_{\alpha}$  of Choquet type 
is positively spanned by $\lambda_i$ for $i = 1,2,3,\dots $, 
(2) implies that $\varphi_{\alpha}(a) \leq \varphi_{\alpha}(b)$. 

(2)$\Rightarrow$(3)  Assume that $\lambda_i(a) \leq \lambda_i(b)$ for any $i = 1, 2,3, \dots $.  
Then there exist constants $d_i$ with $0 \leq d_i \leq 1$ such that $\lambda_i(a) = d_i\lambda_i(b)d_i$. 
 By diagonalization, there exist 
unitaries $u$ and $v$ in  $B(H)$ such that 
$$
uau^* = {\rm diag}(\lambda_1(a),\lambda_2(a),\lambda_3(a),\dots) \oplus O
$$
and
$$
 vbv^* = {\rm diag}(\lambda_1(b),\lambda_2(b),\lambda_3(b),\dots) \oplus O'. 
$$
Let $d = {\rm diag}(d_1,d_2,d_3,\dots) \oplus O'$ be a daigonal operator in $B(H)$. 
There exists a partial isometry $w$ such that 
$uau^* = wdvbv^*d^*w^*$. Then $ c:= u^*wdv$ is a contraction and $a = cbc^*$. 

(3)$\Rightarrow$(2)  Assume that there exists a contraction $c \in B(H)$ such that $a = cbc^*$. 
Then,  for any $i = 1,2, 3,\dots $,
$$
\lambda_i(a) \leq \| c \|  \lambda_i(b) \| c^* \| \leq \lambda_i(b). 
$$
%\color{red}
%\begin{align*}
%  \lambda_i(a) & = \inf_{\dim M^\perp=i-1} \sup \{ (a\xi, \xi) : \|\xi\|=1, \xi \in M \}  \\
%       & = \inf_{\dim M^\perp=i-1} \sup \{ (cbc^*\xi, \xi) : \|\xi\|=1, \xi \in M \}  \\
%       & = \inf_{\dim M^\perp=i-1} \sup \{ (bc^*\xi, c^*\xi) : \|\xi\|=1, \xi \in M \}  \\
%       & \le \inf_{\dim M^\perp=i-1} \sup \{ (b\xi, \xi) : \|\xi\|=1, \xi \in M \} = \lambda_i(b) .
%\end{align*}
%\color{black}
\end{proof}

If ${\rm Tr}$ is the usual linear trace, then $||a||_1:= {\rm Tr}(|a|)$ is the trace norm for a trace class operator
$a$. We shall replace 
the usual linear trace by non-linear traces of Choquet type.  

%%%%%%%%%%%%%%%%%%%%%%%%%%%%%%%%%
%    Definition 3.2
%%%%%%%%%%%%%%%%%%%%%%%%%%%%%%%%%
\begin{definition} \rm
 Let $\varphi = \varphi_{\alpha}$ be  a non-linear trace of Choquet type associated with  
a monotone increasing function $\alpha: {\mathbb N}_0  \rightarrow [0, \infty)$  with $\alpha(0) = 0$. 
Define $|||a|||_{\alpha}:= \varphi_{\alpha}(|a|)$ for $a \in K(H)$ admitting $+\infty$. Since $\varphi_{\alpha}$ 
is unitarily invariant, $|||uav|||_{\alpha} = |||a|||_{\alpha}$ for any unitaries $u$ and $v$ in $B(H)$. 
\end{definition}

%%%%%%%%%%%%%%%%%%%%%%%%%%%%%%%%%%
%     Example 3.3
%%%%%%%%%%%%%%%%%%%%%%%%%%%%%%%%%%
\begin{example} \rm
If $\alpha = (0,1,2,\dots,k,k,k, \dots)$, then $\varphi_{\alpha}(a) =  \sum_{i=1} ^{k}  \lambda_i(a)$. 
Therefore  
$$
|||a|||_{\alpha} = \sum_{i=1} ^{k}  \lambda_i(|a|) = \sum_{i=1} ^{k}  s_i(a)
$$
gives a Ky Fan norm  $\|a\|_{(k)} $, where  $s_i(a) :=  \lambda_i(|a|)$ is the $i$-th singular value of 
$a \in K(H)$. 
\end{example}

%%%%%%%%%%%%%%%%%%%%%%%%%%%%%
%     Proposition 3.4
%%%%%%%%%%%%%%%%%%%%%%%%%%%%%
\begin{proposition} 
Let $\varphi = \varphi_{\alpha}$ be  a non-linear trace of Choquet type associated with  
a monotone increasing function $\alpha: {\mathbb N}_0  \rightarrow [0, \infty)$  with $\alpha(0) = 0$ and 
$\alpha(1) > 0$.  Put $c_i := \alpha(i) - \alpha(i - 1)$ for $i =1,2,3, \ldots$. Recall that 
$ \varphi_{\alpha} = \sum _{i=1}^{\infty}  c_i \lambda_i$. 
Define $|||a|||_{\alpha}:= \varphi_{\alpha}(|a|)$ for $a \in K(H)$. 
Then the following conditions are equivalent: 
\begin{enumerate}
\item[$(1)$] $\alpha$ is concave in the sense that 
$\frac{\alpha(i +1) + \alpha(i - 1)}{2} \leq \alpha(i), \;  (i =1,2,3, $ $\ldots)$.
\item[$(2)$] $(c_i)_i$ is a decreasing sequence: $c_1 \geq c_2 \geq \cdots \geq c_n \geq \cdots $.
\item[$(3)$] $||| \ |||_{\alpha}$ satisfies the triangle inequality: for any $a,b \in K(H) $, 
$|||a + b|||_{\alpha} \leq |||a |||_{\alpha} + ||| b|||_{\alpha}$ admitting $+\infty$. 
\end{enumerate}
\end{proposition}
\begin{proof}
It is trivial that (1) $\Leftrightarrow$ (2). 

(2) $\Rightarrow$ (3): Suppose that $(c_i)_i$ is a decreasing sequence. 
The Ky Fan $k$-norms is defined as 
$$
|| a ||_{(k)} := s_1(a) + s_1(a) + \dots + s_k(a) = \lambda_1(|a|) + \lambda_2(|a|)+ \dots + \lambda_k(|a|))
$$ 
for $a \in K(H)$. 
Then 
$$
|||a|||_{\alpha}= \varphi_{\alpha}(|a|) = \sum_{i =1}^{\infty}  c_i \lambda_i(|a|) = 
\sum_{i =1}^{\infty}( c_i  - c_{i+1})|| a ||_{(i)} .
$$
Therefore $||| \ |||_{\alpha}$ satisfies the triangle inequality.

(3) $\Rightarrow$ (2):Let $I = \sum _{i=1}^{\infty} p_i $ be a resolution of the identity by minimal projections. 
For  $i = 2,3,\dots, $, let  
$$
a = p_1 + p_2 + 2p_3 + 2p_4 + \dots + 2p_{i+1} .
$$
Then 
\begin{align*}
\alpha(i+1) + & \alpha(i-1)  =  \varphi_{\alpha}(a)
 =||| p_1 + p_2 + 2p_3 + 2p_4 + \cdots + 2p_{i+1}|||_{\alpha} \\
 & \leq  
|||p_1+ p_3 + \cdots + p_{i+1}|||_{\alpha} + |||p_2+ p_3 + \cdots + p_{i+1}|||_{\alpha} \\
& = \alpha(i) + \alpha(i) = 2\alpha(i).
\end{align*}
For $i = 1$, we have that 
$$
\alpha(2) + \alpha(0) = \alpha(2) = ||| p_1 + p_2|||_{\alpha} \le |||p_1|||_{\alpha} +|||p_2|||_{\alpha} = 2\alpha(1).
$$
\end{proof}

%%%%%%%%%%%%%%%%%%%%%%%%%%%%%%%%%
%  Remark 3.5
%%%%%%%%%%%%%%%%%%%%%%%%%%%%%%%%%
\begin{remark} If  $\alpha$ is concave, then for any $m, n \in {\mathbb N}_0$, 
$$
\alpha(m + n) \leq \alpha(m) + \alpha(n).
$$
In fact, 
$$
\alpha(m + n) - \alpha(m) = \sum_{i=m +1}^{m+n} c_i \leq \sum_{i=1}^{n} c_i = \alpha(n).
$$
\end{remark}

%%%%%%%%%%%%%%%%%%%%%%%%%%%%%%%%%
%    Definition 3.6
%%%%%%%%%%%%%%%%%%%%%%%%%%%%%%%%%
\begin{definition} \rm
 Let $\varphi = \varphi_{\alpha}$ be  a non-linear trace of Choquet type associated with  
a monotone increasing function $\alpha: {\mathbb N}_0  \rightarrow [0, \infty)$  with $\alpha(0) = 0$ and 
$\alpha(1) > 0$.  
Define $|||a|||_{\alpha}:= \varphi_{\alpha}(|a|)$ for $a \in K(H)$ admitting $+\infty$.  Then 
$a$ is said to be a {\it trace class operator for a non-linear trace}  $\varphi_{\alpha}$  {\it of Choquet type} if 
$|||a|||_{\alpha}= \varphi_{\alpha}(|a|) < \infty$.  The weighted trace class ${\mathcal C}^{\alpha}_1(H)$ is defined as 
 the set of all trace class operators for a non-linear trace $\varphi_{\alpha}$.   
 \end{definition}

%%%%%%%%%%%%%%%%%%%%%%%%%%%%%
%     Theorem 3.7
%%%%%%%%%%%%%%%%%%%%%%%%%%%%%
\begin{theorem} 
Let $\varphi = \varphi_{\alpha}$ be  a non-linear trace of Choquet type associated with  
a monotone increasing function $\alpha: {\mathbb N}_0  \rightarrow [0, \infty)$  with $\alpha(0) = 0$ and 
$\alpha(1) > 0$. Assume that  $\alpha$ is concave.
Then the weighted trace class ${\mathcal C}^{\alpha}_1(H)$ of
trace class operators for a non-linear trace $\varphi_{\alpha}$ 
is a Banach space with respect to the norm $||| \ |||_{\alpha}$.  Moreover  ${\mathcal C}^{\alpha}_1(H)$ 
is a $*$-ideal of $B(H)$.   
\end{theorem}
\begin{proof}
Since  $\alpha$ is concave, $|||a|||_{\alpha}= \varphi_{\alpha}(|a|)$ define a norm on  ${\mathcal C}^{\alpha}_1(H)$. 
Let $(a_n)_n$ be a Cauchy sequence in ${\mathcal C}^{\alpha}_1(H)$. Since the operator norm 
$\|a\| = \lambda_1(|a|) \leq \frac{1}{c_1} ||| a |||_{\alpha}$,  $(a_n)_n$ is also a Cauchy sequence in the 
operator norm in $B(H)$. Therefore there exists a bounded operator $a \in B(H)$ such that 
$\|a_n - a \| \rightarrow 0$.  Since each $a_n$ is a compact operator, $a$ is also a compact operator.  
We shall show that $(a_n)_n$ converges to $a$ in the norm  $||| \ |||_{\alpha}$. Since 
$(a_n)_n$ be a Cauchy sequence in ${\mathcal C}^{\alpha}_1(H)$, for any $\epsilon > 0$ there exists a natural 
number $N$ such that for any natural number $n$ and $m$, if $n \geq N$ and $m \geq N$, then 
$$
\sum_{k=1}^{\infty} \lambda_k(|a_n -a_m|)c_k < \epsilon .
$$
Let $m \rightarrow \infty$.  Since $\|a_m - a \| \rightarrow 0$ in operator norm, $|a_n - a_m|$ coverges 
to $|a_n - a|$ in operator norm. Since each $ \lambda_k$ is operator norm continuous, 
$$
\sum_{k=1}^{\infty} \lambda_k(|a_n -a|)c_k \leq \epsilon .
$$
Since 
$$
|||a|||_{\alpha} \leq |||a - a_n |||_{\alpha} + |||a_n|||_{\alpha} 
\leq  \epsilon +   |||a_n|||_{\alpha}  <  \infty, 
$$
we have $a \in {\mathcal C}^{\alpha}_1(H)$. 
Moreover the above also shows that $(a_n)_n$ converges to $a$ in the norm  $||| \ |||_{\alpha}$. Thus 
the weighted trace class ${\mathcal C}^{\alpha}_1(H)$ is a Banach space.

For any $a \in K(H)$ and $b \in B(H)$,  we have $\lambda_k(|ab|) \leq \lambda_k(|a|) ||b||$, 
$\lambda_k(|ba|) \leq ||b|| \lambda_k(|a|)$ and $\lambda_k(|a^*|) =\lambda_k(|a|)$.  
Therefore $|||ab|||_{\alpha} \leq |||a|||_{\alpha}||b||$, $|||ba|||_{\alpha} \leq ||b|| |||a|||_{\alpha}$ 
and $|||a^*|||_{\alpha} = |||a|||_{\alpha}$.  Hence ${\mathcal C}^{\alpha}_1(H)$ 
is a $*$-ideal of $B(H)$.   
\end{proof}

%%%%%%%%%%%%%%%%%%%%%%%%%%%%%%%%%
%    Definition 3.8
%%%%%%%%%%%%%%%%%%%%%%%%%%%%%%%%%
\begin{definition} \rm
 Let $\varphi_{\alpha}$ be  a non-linear trace of Choquet type associated with  
a monotone increasing function $\alpha: {\mathbb N}_0  \rightarrow [0, \infty)$  with $\alpha(0) = 0$ and 
$\alpha(1) > 0$.  Assume that  $\alpha$ is concave. We shall extend $\varphi_{\alpha}$ to 
the weighted trace class ${\mathcal C}^{\alpha}_1(H)$ as follows: Let $a \in {\mathcal C}^{\alpha}_1(H)$ be
a  trace class operator for the non-linear trace $\varphi_{\alpha}$. Consider the decomposition 
$$
a = \frac{1}{2}(a + a^*) + i  \ \frac{1}{2i}(a - a^*) = a_1 -a_2 + i(a_3 -a_4)
$$
with $a_1, a_2, a_3, a_4 \in (K(H))^+$ and $a_1a_2 = a_3a_4 = 0$.  
Since $\varphi_{\alpha}(|a|) < \infty$ and $\alpha$ is concave, we have
$\varphi_{\alpha}(|a + a^*|) < \infty$ and $\varphi_{\alpha}(|a - a^*|) < \infty$. 
Because 
$$
a_1 \leq a_1 + a_2 = \frac{1}{2}|a + a^*|, a_2 \leq a_1 + a_2 = \frac{1}{2}|a + a^*|, 
$$
$$
a_3 \leq a_3 + a_4 = \frac{1}{2}|a - a^*|, a_4 \leq a_3 + a_4 = \frac{1}{2}|a - a^*|, 
$$
the monotonity of $\varphi_{\alpha}$ implies that 
$$
0 \leq \varphi_{\alpha}(a_i) < \infty, \ \ \ (i = 1,2,3,4). 
$$
Therefore we can define a non-linear trace 
$\varphi_{\alpha} : {\mathcal C}^{\alpha}_1(H) \rightarrow  {\mathbb C}$ by 
$$
\varphi_{\alpha}(a) := \varphi_{\alpha}(a_1) - \varphi_{\alpha}(a_2) + i( \varphi_{\alpha}(a_3) - \varphi_{\alpha}(a_4)),
$$
where we use the same symbol $\varphi_{\alpha}$ for the extended non-linear trace. 
\end{definition}

%%%
%%%   Section 4
%%%
\section{Schatten-von Neumann class for non-linear traces of Choquet type}

Fix a positive number $p \geq 1$.  In this section we shall study general Schatten-von Neumann class ${\mathcal C}^{\alpha}_p(H)$ for non-linear traces of Choquet type on the algebra of compact operators.   First we need to 
consider the case of matrix algebras. 

%%%%%%%%%%%%%%%%%%%%%%%%%%%%%
%     Definition 4.1
%%%%%%%%%%%%%%%%%%%%%%%%%%%%%

\begin{definition} \rm  (see \cite {nagisawatatani2})
Let $\alpha: \{0,1,2, \dots, n\} \rightarrow [0, \infty)$ be a 
monotone increasing function with $\alpha(0) = 0$. We denote by 
$\mu_{\alpha}$ the associated permutation invariant 
monotone measure on $\Omega  = \{1,2, \dots, n\}$. 
Let 
$\lambda(a) = (\lambda_1(a),\lambda_2(a),\dots, \lambda_n(a))$ be the list of the 
eigenvalues of $a \in (M_n({\mathbb C}))^+ $ in decreasing order :
$\lambda_1(a) \geq \lambda_2(a) \geq \dots \geq \lambda_n(a)$ with 
counting multiplicities. 
Define 
$\varphi_{\alpha} : (M_n({\mathbb C}))^+ \rightarrow  {\mathbb C}^+$ 
as follows:
\begin{align*}
\varphi_{\alpha}(a) &= \sum_{i=1} ^{n-1} ( \lambda_i(a)- \lambda_{i+1}(a))\mu_{\alpha}(A_i) 
+ \lambda_n (a) \mu_{\alpha}(A_n)\\
&= \sum_{i=1} ^{n-1} ( \lambda_i(a)- \lambda_{i+1}(a)) \alpha( ^{\#}A_i) 
+ \lambda_n (a) \alpha(^{\#}A_n)\\
& =  \sum_{i=1} ^{n-1} ( \lambda_i(a)- \lambda_{i+1}(a)) \alpha(i) 
+ \lambda_n (a) \alpha(n) ,
\end{align*}
where $A_i = \{1,2,\dots,i \}$.
We call $\varphi_{\alpha}$ the non-linear trace of Choquet type on the matrix algebra $(M_n({\mathbb C}))^+$ 
associated with $\alpha$.  We note that 
\begin{align*}
\varphi_{\alpha}(a) &= \sum_{i=1} ^{n-1} ( \lambda_i(a)- \lambda_{i+1}(a)) \alpha(i) 
+ \lambda_n (a) \alpha(n) \\
& = \sum_{i=2} ^{n}  \lambda_i(a)( \alpha(i) -  \alpha(i-1)) +  \lambda_1(a)\alpha(1) \\
& = \sum_{i=2} ^{n}  c_i \lambda_i(a)  + c_1 \lambda_1(a), 
\end{align*}
where  $\alpha$ and $c_i$ ($i=1,\dots,n)$  are related by 
$c_1 = \alpha(1) = \alpha(1) - \alpha(0) $, $c_i = \alpha(i) - \alpha(i-1)$  $(i=2,3,\ldots,n)$  or  $\alpha(j)= \sum_{i=1}^j c_i$  $(j=1,2,\ldots,n)$. 
Note that $\varphi_{\alpha}$ is norm continuous on $(M_n({\mathbb C}))^+$, since 
each $\lambda_i$ is norm continuous on $(M_n({\mathbb C}))^+$. 
Define $|||a|||_{\alpha,p}:= \varphi_{\alpha}(|a|^p)^{1/p}$ for $a \in M_n({\mathbb C})$. 
\end{definition}

We shall consider when $|||a|||_{\alpha,p}$ satisfies the triangle inequality. We prepare a lemma. 

%%%%%%%%%%%%%%%%%%%%%%%%%%%%%
%   Lemma 4.2
%%%%%%%%%%%%%%%%%%%%%%%%%%%%%
\begin{lemma}\label{lemma:norm-norm}
Let $V$ be a vector space over ${\mathbb C}$ and $\nu_1,\nu_2,\dots,\nu_n$ be norms on $V$.  Let $\beta$ be a 
norm on ${\mathbb R}^n$.  Assume that for any $x =(x_i)_i, y =(y_i)_i \in {\mathbb R}^n$, $0 \leq x \leq y$ 
implies $\beta(x) \leq \beta(y)$, where $0 \leq x \leq y$ means that $0 \leq x_i \leq y_i$ for $i = 1,\dots,n$. 
Define $\gamma: V \rightarrow {\mathbb R}$ by $\gamma(a) = \beta((\nu_1(a), \dots, \nu_n(a)))$ for $a \in V$.  
Then $\gamma$ is a norm on $V$. 
\end{lemma}
\begin{proof}
Since each $\nu_i$ is a norm, for any $a,b \in V$, we have that 
$0 \leq (\nu_1(a+b), \dots, \nu_n(a+b)) \leq (\nu_1(a) + \nu_1(b), \dots, \nu_n(a) + \nu_n(b ) ).$
Hence 
\begin{align*}
    \gamma (a+b) &=\beta((\nu_1(a+b), \dots, \nu_n(a+b))) \\
         & \leq \beta((\nu_1(a) + \nu_1(b), \dots, \nu_n(a) + \nu_n(b))).
\end{align*} 
%\beta( (\nu_1(a), \dots, \nu_n(a)) + (\nu_1(b), \dots, \nu_n(b)))). 
Since $\beta$ is also a norm, 
$$
\gamma (a+b) \leq \beta( (\nu_1(a), \dots, \nu_n(a))) +\beta((\nu_1(b), \dots, \nu_n(b))) 
= \gamma(a) + \gamma(b).
$$
Thus $\gamma$ satisfies the triangle inequality. The other poroperties are clear. 
\end{proof}

%%%%%%%%%%%%%%%%%%%%%%%%%%%%%
%     Theorem 4.3
%%%%%%%%%%%%%%%%%%%%%%%%%%%%%
\begin{theorem}\label{theorem:matrix-p-norm} 
Let $\varphi_{\alpha}$ be  a non-linear trace of Choquet type associated with  
a monotone increasing function $\alpha: \{0,1,2, \dots, n\} \rightarrow [0, \infty)$  
with $\alpha(0) = 0$ and 
$\alpha(1) > 0$.  Put $c_i := \alpha(i) - \alpha(i - 1)$ for $i =1,2,\dots,n$. Fix $p \geq 1$. 
Define $|||a|||_{\alpha,p}:= \varphi_{\alpha}(|a|^p)^{1/p}$ for $a \in M_n({\mathbb C})$. 
Then the following conditions are equivalent: 
\begin{enumerate}
\item[$(1)$] $\alpha$ is concave in the sense that 
$\frac{\alpha(i +2) + \alpha(i )}{2} \leq \alpha(i+1), \;  (i =0, 1,2,3, \dots,n-2)$.
\item[$(2)$] $(c_i)_i$ is a decreasing sequence: $c_1 \geq c_2 \geq \dots \geq c_n$.
\item[$(3)$] $||| \ |||_{\alpha,p}$ satisfies the triangle inequality. 
\end{enumerate}
In this case, $||| \ |||_{\alpha,p}$ is a unitarily invariant norm. 
\end{theorem}
\begin{proof} 
It is trivial that (1) $\Leftrightarrow$ (2). 

(2) $\Rightarrow$ (3): Suppose that $(c_i)_i$ is a decreasing sequence.  
Recall that Ky Fan $k$-norms is defined as 
$$
|| a ||_{(k)} := s_1(a) + s_1(a) + \dots + s_k(a) = \lambda_1(|a|) + \lambda_2(|a|)+ \dots + \lambda_k(|a|))
$$ 
for $a \in  M_n({\mathbb C})$ and $1 \leq k \leq n$. 
More generally, Ky Fan $p$-$k$ norm is defined as
$$
|| a ||_{p,(k)} := (\lambda_1(|a|^p) + \lambda_2(|a|^p)+ \dots + \lambda_k(|a|^p)))^{1/p},
$$ 
see, for example,  \cite[p.199]{Horn-Johnson}. 
Then 
\begin{align*}
|||a|||_{\alpha,p}& = \varphi_{\alpha}(|a|^p)^{1/p} = (\sum_{k =1}^{n}  c_k \lambda_k(|a|^p))^{1/p} \\
&= 
(\sum_{k =1}^{n-1}( c_k  - c_{k+1})|| a ||_{p,(k)}^p + c_n|| a ||_{p,(n)}^p)^{1/p} .
\end{align*}
Since $(c_i)_i$ is a decreasing sequence, $d_k := c_k  - c_{k+1} \geq 0$ for $k = 1,2, \ldots, n-1$ 
and $d_n := c_n \geq 0$. Define 
$\nu_k(a):= d_k^{1/p}|| a ||_{p,(k)}$.  
Consider $p$-norm $\beta$ on  ${\mathbb R}^n$ by $\beta(x) = (\sum_{k =1}^{n} |x_k|^p)^{1/p}$ for 
$x =(x_i)_i \in {\mathbb R}^n$. Then  for $a \in  M_n({\mathbb C})$, 
$$
|||a|||_{\alpha,p} = \beta((\nu_1(a), \dots, \nu_n(a))).
$$
By Lemma~\ref{lemma:norm-norm}, $|||a|||_{\alpha,p}$  satisfies the triangle inequality. 
Since $\varphi_{\alpha}$ 
is unitarily invariant, $|||uav|||_{\alpha,p} = |||a|||_{\alpha,p}$ for any unitaries $u$ and $v$ in 
$M_n({\mathbb C})$.  Now it is clear that $||| \ |||_{\alpha,p}$ is a unitarily invariant norm.

(3) $\Rightarrow$ (1):  Suppose that $||| \ |||_{\alpha,p}$ satisfies the triangle inequality.  
Let $ p_1 + p_2 + \dots + p_n = I$ be the resolution of the identity $I$ by minimal projections. 
Fix $i = 1,2,3, \dots,n-2. $ 

When  $\alpha(i+2) - \alpha(i+1) = 0 $, it is clear that $\frac{\alpha(i +2) + \alpha(i )}{2} \leq \alpha(i+1)$.
So we assume that $\alpha(i+2)>\alpha(i+1)$. 
Then we can choose $s_0 \in (0, 1]$ such that 
\[    \alpha(i+2)-\alpha(i) > (1+s_0)^p (\alpha(i+1)-\alpha(i)) .  \]
For any $s\in (0,s_0]$ and $t\in [0,1]$ we put
\begin{gather*}
  a = 2(p_1+\cdots+p_i)+(1+s)p_{i+1} +(1-t)p_{i+2} \in M_n(\C)  \\
  \text{ and } c = 2(p_1+\cdots + p_i)+ p_{i+1}+p_{i+2} \in M_n(\C).
\end{gather*}
Then we have
\begin{align*}
  |||a|||_{\alpha,p}^p=(2^p-(1+s)^p)\alpha(i) & + ((1+s)^p-(1-t)^p)\alpha(i+1) \\
                                                      & \qquad +(1-t)^p\alpha(i+2) 
\end{align*}
\[  \text{and  } |||c|||_{\alpha,p}^p = (2^p-1)\alpha(i) + \alpha(i+2).  \]
In the case of $t=0$,
\begin{align*}
   |||a|||_{\alpha,p}^p & =(2^p-(1+s)^p)\alpha(i)+((1+s)^p-1) \alpha(i+1) + \alpha(i+2) \\
     & \ge (2^p-1) \alpha(i) + \alpha(i+2) = |||c|||_{\alpha,p}^p.
\end{align*}
In the case of $t=1$,
\begin{align*}
   |||a|||_{\alpha,p}^p & =(2^p-(1+s)^p)\alpha(i)+(1+s)^p \alpha(i+1) \\
     & = 2^p \alpha(i) + (1+s)^p(\alpha(i+1)- \alpha(i))  \\
     & < 2^p \alpha(i) + \alpha(i+2)-\alpha(i) = |||c|||_{\alpha,p}^p.
\end{align*}
So, for any $s\in (0,s_0)$,  we can choose $t\in [0,1]$  such that
\[  |||a|||_{\alpha,p}^p = (2^p-1)\alpha(i) + \alpha(i+2).  \]
We shall show that $s \leq t$.  
Consider 
\[  b = 2(p_1 + \dots + p_i) + (1-t)p_{i+1} + (1+s)p_{i+2} \in M_n({\mathbb C}). \]
Then 
$|||a|||_{\alpha,p}^p =  |||b|||_{\alpha,p}^p = (2^p - 1)\alpha(i) + \alpha(i+2)$.
Since the norm $||| \cdot |||_{\alpha,p} $ satisfies the triangle inequality,  
\[   |||\frac{a + b}{2}|||_{\alpha,p} \le \frac{1}{2}(|||a|||_{\alpha,p} +|||b|||_{\alpha,p}) = |||a|||_{\alpha,p}.  \]   
Since
\[  \frac{a + b}{2} =  2(p_1 + \dots + p_i) + (1+\frac{s-t}{2}) (p_{i+1} + p_{i+2}) ,  \]  
$|||\frac{a + b}{2}|||_{\alpha,p}^p  \leq |||a|||_{\alpha,p}^p$ implies that 
$$
(2^p-(1+\frac{s-t}{2})^p)\alpha(i) + (1 + \frac{s-t}{2})^p\alpha(i+2) \leq (2^p - 1)\alpha(i) + \alpha(i+2).
$$
Thus we have that 
$$
((1 + \frac{s-t}{2})^p -1)(\alpha(i+2) -\alpha(i)) \leq 0.
$$
Then $\alpha(i+2) >\alpha(i)$ implies that $s \leq t$.  
By the relation
\begin{align*}
    &  |||a|||_{\alpha,p}^p  \\
      = & 2^p\alpha(i) +(1+s)^p(\alpha(i+1)-\alpha(i)) +(1-t)^p(\alpha(i+2)-\alpha(i+1))  \\
      = & (2^p-1)\alpha(i) + \alpha(i+2),
\end{align*}
we have
\begin{align*}
  (2^p-1)\alpha(i) + \alpha(i+2) \le 2^p\alpha(i) & +(1+s)^p(\alpha(i+1)-\alpha(i)) \\
   & +(1-s)^p(\alpha(i+2)-\alpha(i+1)) 
\end{align*}
for all $s\in [0, s_0]$, that is,
\[  \alpha(i)-\alpha(i+2) +(1+s)^p(\alpha(i+1)-\alpha(i)) + (1-s)^p (\alpha(i+2)-\alpha(i+1)) \ge 0  \]
for all  $s\in [0, s_0]$.
Since $(1 + s)^p = 1 +ps + {\mathcal O}(s^2)$, we have 
$$
  ps(2\alpha(i+1) - \alpha(i) - \alpha(i+2)) +  {\mathcal O}(s^2) \geq 0. 
$$
Therefore $2\alpha(i+1) - \alpha(i) - \alpha(i+2) \geq 0$, this means that $\alpha$ is concave. 

The case that $i = 0$ is similarly shown.
\end{proof}

%%%%%%%%%%%%%%%%%%%%%%%%%%%%%%%%%
%    Definition 4.4
%%%%%%%%%%%%%%%%%%%%%%%%%%%%%%%%%
\begin{definition} \rm
Suppose that $\alpha$ is concave. Then for  $a \in  M_n({\mathbb C})$, 
we say that $|||a|||_{\alpha,p}:= \varphi_{\alpha}(|a|^p)^{1/p}$ is the 
Schatten-von Neumann $p$-norm of $a$ for a  non-linear trace $\varphi_{\alpha}$ of Choquet type. 
\end{definition}

Now we shall return to consider the case of the algebra $K(H)$ of compact operators. 
Fix a positive number $p \geq 1$. 
%%%%%%%%%%%%%%%%%%%%%%%%%%%%%%%%%
%    Definition 4.5
%%%%%%%%%%%%%%%%%%%%%%%%%%%%%%%%%
\begin{definition} \rm
 Let $\varphi = \varphi_{\alpha}$ be  a non-linear trace of Choquet type associated with  
a monotone increasing function $\alpha: {\mathbb N}_0  \rightarrow [0, \infty)$  with $\alpha(0) = 0$. 
Define $|||a|||_{\alpha,p}:= \varphi_{\alpha}(|a|^p)^{1/p}$ for $a \in K(H)$ admitting $+\infty$. Since $\varphi_{\alpha}$ 
is unitarily invariant, $|||uav|||_{\alpha,p} = |||a|||_{\alpha,p}$ for any unitaries $u$ and $v$ in $B(H)$. 
\end{definition}

%%%%%%%%%%%%%%%%%%%%%%%%%%%%%%%%%%
%     Example 4.6
%%%%%%%%%%%%%%%%%%%%%%%%%%%%%%%%%%
\begin{example} \rm
If $\alpha = (0,1,2,\dots,k,k,k, \dots)$, then $\varphi_{\alpha}(a) =  \sum_{i=1} ^{k}  \lambda_i(a)$. 
Thus
$$
|||a|||_{\alpha,p} =( \sum_{i=1} ^{k}  \lambda_i(|a|^p) )^{1/p}
$$
gives a Ky Fan $p$-$k$ norm  $||a||_{p,(k)} $ on $K(H)$.  In fact the triangle inequality on $K(H)$ 
can be checked as follows: First assume that $a,b \in K(H) $ are finite rank operators. 
Then there exist a finite $N$-dimensional subspace $M \subset H$ such that $a,b$ and $a+b$ can be regarded 
as operators on $M$.  Since $M$ is a finite $N$-dimensional subspace, 
we can regard $a,b$ and $a+b$ are in $B(M)  \cong M_N({\mathbb C})$.  
Hence we have $|||a + b|||_{\alpha,p} \leq |||a |||_{\alpha,p} + ||| b|||_{\alpha,p}$.  
Since each $\lambda_i$ is operator norm continuous on $K(H)^+$, the map 
$a \mapsto |||a|||_{\alpha,p}$ is operator norm continuous on $K(H)$.  
The set of finite rank operators is dense in  $K(H)$ with respect to the operator norm. 
Therefore the triangle inequality holds on $K(H)$. The other properties of norm are trivial. 
Thus Ky Fan $p$-$k$ norm  $|| \cdot ||_{p,(k)} $ is in fact a norm on $K(H)$.
\end{example}

%%%%%%%%%%%%%%%%%%%%%%%%%%%%%
%     Theorem 4.7
%%%%%%%%%%%%%%%%%%%%%%%%%%%%%
\begin{theorem} 
Let $\varphi = \varphi_{\alpha}$ be  a non-linear trace of Choquet type associated with  
a monotone increasing function $\alpha: {\mathbb N}_0  \rightarrow [0, \infty)$  with $\alpha(0) = 0$ and 
$\alpha(1) > 0$.  Put $c_i := \alpha(i) - \alpha(i - 1)$ for $i =1,2,\dots $. Recall that 
$ \varphi_{\alpha} (a)= \sum _{i=1}^{\infty}  c_i \lambda_i(a)$. 
Define $|||a|||_{\alpha,p}:= \varphi_{\alpha}(|a|^p)^{1/p}$ for $a \in K(H)$. 
Then the following conditions are equivalent: 
\begin{enumerate}
\item[$(1)$] $\alpha$ is concave in the sense that 
$\frac{\alpha(i +1) + \alpha(i - 1)}{2} \leq \alpha(i), \;  (i =1,2,3, $ $\dots)$.
\item[$(2)$] $(c_i)_i$ is a decreasing sequence: $c_1 \geq c_2 \geq \dots \geq c_n \geq \dots $.
\item[$(3)$] $||| \ |||_{\alpha,p}$ satisfies the triangle inequality: for any $a,b \in K(H) $, 
$|||a + b|||_{\alpha,p} \leq |||a |||_{\alpha,p} + ||| b|||_{\alpha,p}$ admitting $+\infty$. 
\end{enumerate}
\end{theorem}
\begin{proof}
It is trivial that (1) $\Leftrightarrow$ (2). 

(2) $\Rightarrow$ (3): Suppose that $(c_i)_i$ is a decreasing sequence.  
Let $c_0$ be the set of sequences vanishing at infinity. 
Define $\beta: c_0 \rightarrow [0,\infty]$ by $\beta(x) = (\sum_{k =1}^{\infty} |x_k|^p)^{1/p}$ for 
$x =(x_k)_k \in c_0$. Then $\beta$ satisfies the triangle inequality admitting $+\infty$. 
Recall that Ky Fan $p$-$k$ norm on $K(H)$ is defined as
$$
|| a ||_{p,(k)} := (\lambda_1(|a|^p) + \lambda_2(|a|^p)+ \dots + \lambda_k(|a|^p)))^{1/p}.
$$ 
Then 
\begin{align*}
|||a|||_{\alpha,p}& = \varphi_{\alpha}(|a|^p)^{1/p} = (\sum_{k =1}^{\infty}  c_k \lambda_k(|a|^p))^{1/p} \\
&= 
(\sum_{k =1}^{\infty}( c_k  - c_{k+1})|| a ||_{p,(k)}^p)^{1/p}.
\end{align*}
Since $(c_i)_i$ is a decreasing sequence, $d_k := c_k  - c_{k+1} \geq 0$ for any $k$.  
Define $\nu_k(a):= d_k^{1/p}|| a ||_{p,(k)}$.  
Then  for $a \in K(H)$
$$
|||a|||_{\alpha,p} = \beta((\nu_1(a), \nu_2(a),\dots)).
$$
For any $x =(x_i)_i, y =(y_i)_i \in c_0$, $0 \leq x \leq y$ 
implies $\beta(x) \leq \beta(y)$, where $0 \leq x \leq y$ means that $0 \leq x_i \leq y_i$ for $i = 1,2,\dots$. 
By a similar proof like Lemma~\ref{lemma:norm-norm}, we can show that  
 $|||a|||_{\alpha,p}$  satisfies the triangle inequality admitting $+\infty$. 

(3) $\Rightarrow$ (1): Just 
apply Theorem~\ref{theorem:matrix-p-norm} and take $n \rightarrow \infty$. 
\end{proof}

%%%%%%%%%%%%%%%%%%%%%%%%%%%%%%%%%
%    Definition 4.8
%%%%%%%%%%%%%%%%%%%%%%%%%%%%%%%%%
\begin{definition} \rm
 Let $\varphi = \varphi_{\alpha}$ be  a non-linear trace of Choquet type associated with  
a monotone increasing function $\alpha: {\mathbb N}_0  \rightarrow [0, \infty)$  with $\alpha(0) = 0$ and 
$\alpha(1) > 0$.  
Define $|||a|||_{\alpha,p}:= \varphi_{\alpha}(|a|^p)^{1/p}$ for $a \in K(H)$ admitting $+\infty$.  Then 
$a$ is said to be a {\it Schatten-von Neumann $p$-class operator for a non-linear trace}  $\varphi_{\alpha}$  {\it of Choquet type} if 
$|||a|||_{\alpha,p}= \varphi_{\alpha}(|a|^p)^{1/p}< \infty$.  The weighted Schatten-von Neumann $p$-class  ${\mathcal C}^{\alpha}_p(H)$ is defined as 
 the set of all Schatten-von Neumann $p$-class  operators for a non-linear trace $\varphi_{\alpha}$.   
 \end{definition}

%%%%%%%%%%%%%%%%%%%%%%%%%%%%%
%     Theorem 4.9
%%%%%%%%%%%%%%%%%%%%%%%%%%%%%
\begin{theorem} 
Let $\varphi = \varphi_{\alpha}$ be  a non-linear trace of Choquet type associated with  
a monotone increasing function $\alpha: {\mathbb N}_0  \rightarrow [0, \infty)$  with $\alpha(0) = 0$ and 
$\alpha(1) > 0$. Assume that  $\alpha$ is concave.
Then the weighted Schatten-von Neumann $p$-class ${\mathcal C}^{\alpha}_p(H)$ 
for a non-linear trace $\varphi_{\alpha}$ 
is a Banach space with respect to the norm $||| \ |||_{\alpha,p}$.  Moreover  ${\mathcal C}^{\alpha}_p(H)$ 
is a $*$-ideal of $B(H)$.   
\end{theorem}
\begin{proof}
Since  $\alpha$ is concave, $|||a|||_{\alpha,p}= \varphi_{\alpha}(|a|^p)^{1/p}$ define a norm on  ${\mathcal C}^{\alpha}_p(H)$. 
Let $(a_n)_n$ be a Cauchy sequence in ${\mathcal C}^{\alpha}_p(H)$. Since the operator norm 
$\|a\| = \lambda_1(|a|) \leq \frac{1}{c_1^{1/p}} ||| a |||_{\alpha,p}$,  $(a_n)_n$ is also a Cauchy sequence in the 
operator norm in $B(H)$. Therefore there exists a bounded operator $a \in B(H)$ such that 
$\|a_n - a \| \rightarrow 0$.  Since each $a_n$ is a compact operator, $a$ is also a compact operator.  
We shall show that $(a_n)_n$ converges to $a$ in the norm  $||| \ |||_{\alpha,p}$. Since 
$(a_n)_n$ be a Cauchy sequence in ${\mathcal C}^{\alpha}_p(H)$, for any $\epsilon > 0$ there exists a natural 
number $N$ such that for any natural number $n$ and $m$, if $n \geq N$ and $m \geq N$, then 
$$
(\sum_{k=1}^{\infty} \lambda_k(|a_n -a_m|^p)c_k )^{1/p} < \epsilon .
$$
Let $m \rightarrow \infty$.  Since $\|a_m - a \| \rightarrow 0$ in operator norm, $|a_n - a_m|$ coverges 
to $|a_n - a|$ in operator norm. Since each $ \lambda_k$ is operator norm continuous, 
$$
(\sum_{k=1}^{\infty} \lambda_k(|a_n -a|^p)c_k )^{1/p} \leq \epsilon .
$$
Since 
$$
|||a|||_{\alpha,p} \leq |||a - a_n |||_{\alpha,p} + |||a_n|||_{\alpha,p}
\leq  \epsilon +   |||a_n|||_{\alpha,p}  <  \infty, 
$$
we have $a \in {\mathcal C}^{\alpha}_p(H)$. 
Moreover the above also shows that $(a_n)_n$ converges to $a$ in the norm  $||| \ |||_{\alpha,p}$. Thus 
the weighted $p$-class  ${\mathcal C}^{\alpha}_p(H)$ is a Banach space.  

For any $a \in K(H)$ and $b \in B(H)$,  we have $\lambda_k(|ab|) \leq \lambda_k(|a|) ||b||$, 
$\lambda_k(|ba|) \leq ||b|| \lambda_k(|a|)$ and $\lambda_k(|a^*|) =\lambda_k(|a|)$.  
Therefore $|||ab|||_{\alpha,p} \leq |||a|||_{\alpha,p}||b||$, $|||ba|||_{\alpha,p} \leq ||b|| |||a|||_{\alpha,p}$ 
and $|||a^*|||_{\alpha,p} = |||a|||_{\alpha,p}$.  Hence ${\mathcal C}^{\alpha}_1(H)$ 
is a $*$-ideal of $B(H)$.   
\end{proof}

%%%%%%%%%%%%%%%%%%%%%%%%%%%%%%%%
%  Proposition 4.10
%%%%%%%%%%%%%%%%%%%%%%%%%%%%%%%%
\begin{proposition} 
Let $\alpha$ and $\beta$ be 
monotone increasing functions $\alpha, \beta: {\mathbb N}_0  \rightarrow [0, \infty)$  with
$\alpha(0)= \beta(0)= 0$, $\alpha(1) > 0$ and $\beta(1) > 0$. 
Consider $p \geq 1$ and $q \geq 1$. Then 
we have the following inclusions among  the weighted Schatten-von Neumann classes:
\begin{enumerate}
\item[$(1)$] If $1 \leq p \leq q$, then ${\mathcal C}^{\alpha}_p(H) \subset {\mathcal C}^{\alpha}_q(H)$. 
\item[$(2)$] If for any $n \in \mathbb N$ , 
$c_n:=\alpha(n) -\alpha(n-1) \leq d_n:= \beta(n) - \beta(n-1)$, then 
we have that ${\mathcal C}^{\beta}_p(H) \subset {\mathcal C}^{\alpha}_p(H)$.
\end{enumerate}
\end{proposition}
\begin{proof}
(1) Recall that for any  $x \in {\mathbb R}$ with $0 \leq x <1$ and $1 \leq p \leq q$, we have that 
$0 \leq x^q \leq x^p <1$.  Take $a \in {\mathcal C}^{\alpha}_p(H)$. Then 
$$
|||a|||_{\alpha,p} = \varphi_{\alpha}(|a|^p)^{1/p} = (\sum_{k =1}^{\infty}  c_k \lambda_k(|a|^p))^{1/p} < \infty.
$$
Since $\lambda_n(|a|) \rightarrow 0$, there exists $N \in {\mathbb N}$ such that $\lambda_n(|a|) < 1$ for any $n\ge N$.  
Therefore 
\begin{align*}
|||a|||_{\alpha,q}^q & = \varphi_{\alpha}(|a|^q) 
= \sum_{k =1}^{\infty}  c_k \lambda_k(|a|^q)\\
& = \sum_{k =1}^{N-1}  c_k \lambda_k(|a|^q)+  \sum_{k =N}^{\infty}  c_k \lambda_k(|a|^q)\\
& \leq (\sum_{k =1}^{N-1}  c_k \lambda_k(|a|^q)) +  \sum_{k =N}^{\infty}  c_k \lambda_k(|a|^p)
< \infty.
\end{align*}
Hence $a \in {\mathcal C}^{\alpha}_q(H)$. \\
(2) Take $a \in {\mathcal C}^{\beta}_p(H)$. Then 
$$
|||a|||_{\alpha,p} = (\sum_{k =1}^{\infty}  c_k \lambda_k(|a|^p))^{1/p} 
\leq  (\sum_{k =1}^{\infty}  d_k \lambda_k(|a|^p))^{1/p}
< \infty .
$$
Hence $a \in {\mathcal C}^{\alpha}_p(H)$.
\end{proof}

%%%
%%%  section 5
%%%
\section{Non-linear traces of Sugeno type}
In this section we study non-linear traces of Sugeno type on the algebra $K(H)$ of  compact operators. 
We recall  the Sugeno integral with respect to a monotone measure on a finite set 
$\Omega  = \{1,2, \dots, n\}$.  

%%%%%%%%%%%%%%%%%%%%%%%%%%%%%%%%%%%
%     Definition 5.1
%%%%%%%%%%%%%%%%%%%%%%%%%%%%%%%%%%%
\begin{definition} \rm 
The  discrete Sugeno integral  of $f = (x_1, x_2,\dots, x_n) \in [0,\infty)^n$ 
with respect to a monotone measure $\mu$ on a finite set 
$\Omega  = \{1,2, \dots, n\}$ is defined as follows:  
$$
{\rm (S)}\int f d\mu = \vee_{i=1}^{n} (x_{\sigma(i) } \wedge \mu(A_i) ) , 
$$
where $\sigma$ is a permutation on $\Omega$ such that  
$x_{\sigma(1)} \geq x_{\sigma(2)} \geq  \dots \geq x_{\sigma(n)}$, 
$A_i = \{\sigma(1),\sigma(2),\dots,\sigma(i)\}$ and 
$\vee =\max$ , $\wedge = \min$. 
Here we should note that 
$$
 f = \vee_{i=1}^{n} (x_{\sigma(i) } \chi_{A_i}) .
$$
\end{definition}

Let $A = {\mathbb C}^n$ and define 
$(\text{{\rm S-}}\varphi)_{\mu} :( {\mathbb C}^n)^+ \rightarrow {\mathbb C}^+$ 
by the  Sugeno integral $(\text{{\rm S-}}\varphi)_{\mu}(f) ={\rm (S)}\int f d\mu$.  
Then $(\text{{\rm S-}}\varphi)_{\mu}$ is a non-linear monotone positive map.

We shall consider a operator version of the discrete Sugeno integral. 

%%%%%%%%%%%%%%%%%%%%%%%%%%%%%
%     Definition  5.2
%%%%%%%%%%%%%%%%%%%%%%%%%%%%%
\begin{definition}
\rm
Let $\alpha: {\mathbb N}_0  \rightarrow [0, \infty)$ be a 
monotone increasing function with $\alpha(0) = 0$, that is, 
$$
0 = \alpha(0) \leq  \alpha(1) \leq  \alpha(2) \leq \cdots \leq  \alpha(n) \leq \cdots .
$$
We denote by 
$\mu_{\alpha}$ the associated permutation invariant 
monotone measure on ${\mathbb N}$. 
Define 
$\psi_{\alpha} : K(H)^+ \rightarrow  [0, \infty]$
as follows:  For $a \in K(H)^+$, \\
let 
$\lambda(a) = (\lambda_1(a),\lambda_2(a),\ldots, \lambda_n(a), \ldots)$ be the list of the 
eigenvalues of $a$ in decreasing order :
$\lambda_1(a) \geq \lambda_2(a) \geq \cdots \geq \lambda_n(a) \geq \cdots $ with 
counting multiplicities, where  the sequence $(\lambda_n(a))_n$ converges to zero.  We have  
$a = \sum_{n=1}^{\infty} \lambda_n(a)p_n$, where each $p_n$ is a one dimensional spectral projection of $a$.  

Let
\begin{align*}
\psi_{\alpha}(a)  &= \lor_{i=1} ^{\infty} ( \lambda_i(a) \land \mu_{\alpha}(A_i)) \\
&= \lor_{i=1} ^{\infty} ( \lambda_i(a) \land \alpha( ^{\#}A_i) ) \\
&= \lor_{i=1} ^{\infty} ( \lambda_i(a) \land \alpha(i)) ,
\end{align*}
where $A_i = \{1,2,\dots,i \}$.
We call $\psi_{\alpha}$ the non-linear trace of Sugeno type on the algebra $K(H)$ of  compact operators 
associated with $\alpha$. 
Note that $\psi_{\alpha}$ is lower semi-continuous on $K(H)^+$ in norm.  
In fact, let 
$$
S_n(a) =  \lor_{i=1} ^{n} ( \lambda_i(a) \land \alpha(i)).
$$
Since 
each $\lambda_i$ is norm continuous on $K(H)^+$,  $S_n: K(H)^+ \rightarrow  [0, \infty)$ 
is also norm continuous on  $K(H)^+$.  Moreover 
$$
\psi_{\alpha}(a)  = \sup \{S_n(a) \ | \ n \in  {\mathbb N} \}.
$$
Therefore $\psi_{\alpha}$ is lower semi-continuous on $K(H)^+$ in norm.

%By mini-max theorem or (Cauchy's Interlacing Theorem),  for any projection $p$ 
By Courant-Fischer's mini-max theorem, for any projection $p$
$$
\lambda_i(pap) \leq \|p\| \lambda_i(a) \|p\| \leq \lambda_i(a).
$$
Therefore we have that 
$$
\psi_{\alpha}(pap) \leq   \psi_{\alpha}(a) .
$$
\end{definition}

%%%%%%%%%%%%%%%%%%%%%%%%%%%%%%%%
%     Example 5.3
%%%%%%%%%%%%%%%%%%%%%%%%%%%%%%%%
\begin{example}
\begin{enumerate}
  \item[(1)] If $\alpha =(0,1,2,3,\dots)$, then  
$\psi_{\alpha}(a) = \lor_{i=1} ^{\infty} ( \lambda_i(a) \land i) $.
Thus $\psi_{\alpha}(a)  \geq n$ if and only if $ \lambda_n(a) \geq n$. 
Hence $\psi_{\alpha}$ is something like h-index. 
  \item[(2)] If $\alpha =(0,r,r,r,\dots)$ for some $r > 0$, then $\psi_{\alpha}(a) =r$ 
if $\lambda_1(a) \geq r$ and $\psi_{\alpha}(a) =  \lambda_1(a)$ if $\lambda_1(a) <r$.
  \item[(3)]  If $\alpha = (0,1,2,\dots,k,k,\dots)$, then $\psi_{\alpha}(a) = k$ if  $\lambda_k(a) \geq k$. 
Assume that $\lambda_k(a) < k$. If  $\lambda_{i+1}(a) < i \leq  \lambda_{i}(a)$, then $\psi_{\alpha}(a) = i$. 
If $i \leq  \lambda_{i}(a)$ and $i \leq  \lambda_{i+1}(a) < i+1$, then $\psi_{\alpha}(a) = \lambda_{i+1}(a)$. 
\end{enumerate}
\end{example}

%%%%%%%%%%%%%%%%%%%%%%%%%%%%%
%     Lemma 5.4
%%%%%%%%%%%%%%%%%%%%%%%%%%%%%

\begin{lemma}
If $p$ is a finite rank projection and $c$ is a positive constant,  then $\psi_{\alpha}(cp) = c \land \alpha(\dim p)$.
\end{lemma}
\begin{proof}
Let $n = \dim p$.  Then $\psi_{\alpha}(cp) = \lor_{i=1} ^{n} ( \lambda_i(cp) \land \alpha(i)) $.
Assume that $c \leq \alpha(n)$.  If $i < n$, then $\lambda_i(cp) \land \alpha(i) = c \land \alpha(i) \leq c$. 
If $i = n$, then $\lambda_n(cp) \land \alpha(n) = c \land \alpha(n) =c$. Therefore $\psi_{\alpha}(cp) = c$. 
Assume that $c >  \alpha(n)$.  Then $\psi_{\alpha}(cp) = \lor_{i=1} ^{n} ( \alpha(i)) =  \alpha(n)$. 
Therfore  $\psi_{\alpha}(cp) = c \land \alpha(\dim p)$.
\end{proof}

%%%%%%%%%%%%%%%%%%%%%
%% Definition 5.5
%%%%%%%%%%%%%%%%%%%%%
\begin{definition} \rm 
Let $\psi : (K(H))^+ \rightarrow  {\mathbb C}^+$  be a non-linear positive map.
\begin{itemize}
  \item $\psi$ is {\it positively F-homogeneous} if 
$\psi(kI \land a) = k \land \psi(a)$ for any $a \in  (K(H))^+$ and any scalar $k \geq 0$.
  \item $\psi$ is {\it comonotonic F-additive on the spectrum} if 
$$
\psi((f\lor g) (a)) = \psi(f(a)) \lor \psi(g(a)) 
$$  
for any $a \in  (K(H))^+$ and 
any comonotonic functions $f$ and $g$   in $C(\sigma(a))^+$  with $f(0) = g(0) = 0$, where 
$f(a)$ is a functional calculus of $a$ by $f$. 
  \item $\psi$ is {\it monotonic increasing F-additive on the spectrum} if 
$$
\psi((f \lor g)(a)) = \psi(f(a)) \lor  \psi(g(a))
$$  
for any $a \in  (K(H))^+$ and 
any monotone increasing functions $f$ and $g$  in $C(\sigma(a))^+$  with $f(0) = g(0) = 0$. 
Then by induction, we also have 
\[   \psi(\bigvee_{i=1}^n f_i(a)) = \bigvee_{i=1}^n \psi(f_i(a))  \]
%$$
%\psi(f_1(a) \lor  f_2(a) \lor  \dots \lor f_n(a)) = \psi(f_1(a)) \lor  \psi(f_2(a)) \lor  \dots \lor \psi(f_n(a))
%$$
for any monotone increasing functions $f_1, f_2, \dots, f_n$ in  $C(\sigma(a))^+$ with 
$f_i(0) = 0$ for $i=1,2,\dots,n$.
\end{itemize}
\end{definition}

We shall characterize non-linear traces of Sugeno type on  $(K(H))^+$. 
%%%%%%%%%%%%%%%%%%%%%%%%%%%%%
%     Theorem 5.6
%%%%%%%%%%%%%%%%%%%%%%%%%%%%%
\begin{theorem} 
Let $\psi : (K(H))^+ \rightarrow  {\mathbb C}^+$ be a non-linear 
positive map.  Then the following are equivalent:
\begin{enumerate}
\item[$(1)$]  $\psi$ is a non-linear trace $\psi = \psi_{\alpha}$  of Sugeno type associated with  
a monotone increasing function $\alpha: {\mathbb N}_0 \rightarrow [0, \infty)$  with $\alpha(0) = 0$.
\item[$(2)$] $\psi$ is monotonic increasing F-additive on the spectrum, unitarily invariant, monotone,  
positively F-homogeneous, lower semi-continuous on $K(H)^+$ in norm 
and for any finite rank positive operator $a \in F(H)^+$, 
$$\lim_{c\rightarrow \infty} \psi(ca) < +\infty .$$ 
\item[$(3)$] $\psi$ is comonotonic F-dditive on the spectrum, unitarily invariant, monotone and 
positively F-homogeneous, lower semi-continuous on $K(H)^+$ in norm 
and for any finite rank positive operator $a \in F(H)^+$, 
$$\lim_{c\rightarrow \infty} \psi(ca) < +\infty .$$  
\end{enumerate}
\end{theorem}
\begin{proof}
(1)$\Rightarrow$(2) Assume that $\psi$ is a non-linear trace $\psi = \psi_{\alpha}$  of Sugeno type associated with  
a monotone increasing function $\alpha$. 
For $a \in K(H)^+$, 
let 
$\lambda(a) = (\lambda_1(a),\lambda_2(a),\dots, \lambda_n(a), \dots)$ be the 
eigenvalue sequence of $a$. We have  
$a = \sum_{n=1}^{\infty} \lambda_n(a)p_n$, where each $p_n$ is a one dimensional spectral projection of $a$. 
Then the spectrum $\sigma(a) = \{\lambda_i(a) \ | \ i = 1,2, \dots \} \cup \{0\}$. 
For any monotone increasing functions $f$ and $g$  in $C(\sigma(a))^+$ with $f(0) = g(0) = 0$, 
we have the  spectral decompositions 
\begin{gather*}
f(a) = \sum_{i=1}^{\infty} f(\lambda_i(a)) p_i, \ \ \ g(a) = \sum_{i=1}^{\infty} g(\lambda_i(a)) p_i,   \\
\text{and } (f \lor g)(a) = \sum_{i=1}^{\infty}(f(\lambda_i(a)) \lor (g(\lambda_i(a)))) p_i. 
\end{gather*}
Since $f \lor g$ is also monotone increasing on $\sigma(a)$ as well as $f$ and $g$, we have 
\begin{gather*}
f(\lambda_1(a)) \geq f(\lambda_2(a)) \geq \cdots \geq f(\lambda_n(a)) \geq \cdots,    \\
g(\lambda_1(a)) \geq g(\lambda_2(a)) \geq \cdots \geq g(\lambda_n(a))\geq \cdots,   \\
\text{and }  (f \lor g)(\lambda_1(a)) \geq (f \lor g)(\lambda_2(a)) \geq \cdots \geq (f \lor g)(\lambda_n(a))\geq \cdots.
\end{gather*}
Since $f$ and $g$  in $C(\sigma(a))$ with $f(0) = g(0) = 0$,  the sequences  $(f(\lambda_n(a)))_n$,  
$(g(\lambda_n(a)))_n$ and $((f \lor g)(\lambda_n(a)))_n$ converge to $0$. Hence  $f(a), g(a)$  and 
$f(a) \lor g(a) := (f \lor g)(a)$ are positive compact operators.   
Therefore
\begin{align*}
& \psi_{\alpha}((f \lor g)(a))  = \lor_{i=1} ^{\infty} ( (f \lor g)(\lambda_i(a)) \land \alpha(i))\\
 = & \lor_{i=1} ^{\infty} (((f (\lambda_i(a)) \lor g (\lambda_i(a)) ) \land \alpha(i))\\
 = & \lor_{i=1} ^{\infty} (((f (\lambda_i(a)) \land \alpha(i))  \lor (g(\lambda_i(a)) \land \alpha(i)) )  \\
 = & (\lor_{i=1} ^{\infty} ( (f (\lambda_i(a)) \land \alpha(i)))  \lor ( \lor_{i=1} ^{\infty} ( (g (\lambda_i(a)) \land \alpha(i)))\\
 = & \psi_{\alpha}(f(a)) \lor  \psi_{\alpha}(g(a)) . 
\end{align*}
Thus $\psi_{\alpha}$ is monotonic increasing F-additive on the spectrum. 

For $a,b \in K(H)^+$, let 
$\lambda(a) = (\lambda_1(a),\lambda_2(a),\dots, \lambda_n(a), \dots)$ and 
$\lambda(b) = (\lambda_1(b),\lambda_2(b),\dots,$ $\lambda_n(b), \dots )$ be  the 
eigenvalue sequences of $a$ and $b$. 
Assume that $a \leq b$. 
By the mini-max  principle for eigenvalues, we have that 
$\lambda_i(a) \leq \lambda_i(b)$ for $i = 1,2,\dots$.  
Hence 
\begin{align*}
\psi_{\alpha}(a) & = \lor_{i=1} ^{\infty} ( \lambda_i(a) \land \alpha(i)) \\
& \leq \lor_{i=1} ^{\infty} ( \lambda_i(b) \land \alpha(i)) = \psi_{\alpha}(b).
\end{align*}
Thus $\psi_{\alpha}$ is monotone. 

For a positive scalar $k$ , $kI \land a = \sum_{i=1}^{\infty}(k \land \lambda_i(a)) p_i$ and 
$$
\lambda(kI \land a) = (k\land \lambda_1(a),k\land \lambda_2(a),\dots,k \land \lambda_n(a), \dots )
$$ 
is the list of the 
eigenvalues of $ka$ in decreasing order, hence 
\begin{align*}
\psi_{\alpha}(k \land a) &= \lor_{i=1} ^{\infty} ( kI \land \lambda_i(a) \land \alpha(i)) \\
 &= k \land (\lor_{i=1} ^{\infty} (\lambda_i(a) \land \alpha(i)) ) 
 = k  \land \psi_{\alpha}(a). 
\end{align*}
Thus $\psi_{\alpha}$ is positively F-homogeneous. 
It is clear that $\psi_{\alpha}$ is unitarily invariant by definition. 
 We already showed that $\varphi_{\alpha}$ is 
lower semi-continuous on $K(H)^+$ in norm. 
Moreover for any finite rank positive operator $a \in F(H)^+$ such that 
$a = \sum_{i=1}^n \lambda_i(a) p_i$
be the spectral decomposition of $a$, where 
$\lambda(a) = (\lambda_1(a),\lambda_2(a),\dots, \lambda_n(a), 0, \dots)$ 
is the list of the 
eigenvalues of $a$ in decreasing order with counting multiplicities with $\lambda_n(a) \not= 0.$ 
Then 
\begin{align*}
\lim_{c\rightarrow \infty} \psi_{\alpha}(ca) 
 & = \lim_{c\rightarrow \infty} \lor_{i=1} ^{n} (c \lambda_i(a) \land \alpha(i))  \\
 & = \lim_{c\rightarrow \infty} \lor_{i=1} ^{n} (\alpha(i))
    = \alpha (n) < +\infty. 
\end{align*}

(2)$\Rightarrow$(1)  
Assume that $\psi$ is monotonic increasing F-additive on the spectrum, unitarily invariant, monotone,  
positively F-homogeneous. We also assume that $\varphi $ is lower semi-continuous on $K(H)^+$ in norm and 
$\lim_{c\rightarrow \infty} \psi(ca) < +\infty $ for any finite rank operator $a \in F(H)^+$.
Let $I = \sum_n^{\infty} p_n$ be the decomposition of the identity by 
minimal projections. Define  a function $\alpha: {\mathbb N}_0 \rightarrow [0, \infty)$  with $\alpha(0) = 0$
by 
$$
\alpha (i) =\lim_{c\rightarrow \infty} \psi(c(p_1+ p_2 + \dots + p_i))  <\infty. 
$$
Since  $\varphi $ is unitarily invariant, $\alpha$ does not depend 
on the choice of minimal projections. Since  $\varphi $ is monotone, $\alpha$ is monotone increasing. 
For a positive finite rank operator $b \in F(H)^+$, 
let 
$b = \sum_{i=1}^{n} \lambda_i(b) q_i$
be the spectral decomposition of $b$, where 
$\lambda(b) = (\lambda_1(b),\lambda_2(b),\dots, \lambda_n(b), 0, \dots)$ is the  
eigenvalue sequence of $b$.  
Define  functions $f, f_1, f_2, \dots, f_n \in C(\sigma(b))^+$  by $f(x) =x$ and for $i = 1,2,\dots,n$
$$
f_i = \lambda_i(b) \chi_{\{\lambda_1(b),\lambda_2(b),\dots, \lambda_i(b)\}} = \lambda_i(b)I \land 
c \chi_{\{\lambda_1(b),\lambda_2(b),\dots, \lambda_i(b)\}}
$$
for sufficient large $c \geq \lambda_1(b) = \|b\|$, which does not depend on such $c$. 
Each $f_i$ is monotone increasing function on $\sigma(b)$. Since 
$f = \lor_{i=1} ^{n} f_i$, we have that 
$
b = (\lor_{i=1} ^{n} f_i)(b). 
$
Since $\psi$ is monotonic increasing F-additive on the spectrum and positively F-homogeneous, 
we have that 
\begin{align*}
&  \psi(b) = \psi((\lor_{i=1} ^{n} f_i)(b))  = \lor_{i=1} ^{n}\psi(f_i(b)) \\
= &  \lor_{i=1} ^{n}\psi((\lambda_i(b)I \land 
c \chi_{\{\lambda_1(b),\lambda_2(b),\dots, \lambda_i(b)\}})(b)) \\
= &   \lor_{i=1} ^{n}(\lambda_i(b) \land 
\psi((c \chi_{\{\lambda_1(b),\lambda_2(b),\dots, \lambda_i(b)\}})(b)) \\
= &  \lor_{i=1} ^{n} (\lambda_i(b) \land \psi(c(q_1 + q_2 + \dots q_i)) \\
= &  \lor_{i=1} ^{n} (\lambda_i(b) \land (\lim_{c\to \infty}\psi(c(q_1 + q_2 + \dots q_i)) \\
= &  \lor_{i=1} ^{n} (\lambda_i(b) \land \alpha(i)) = \psi_{\alpha}(b).
\end{align*}
Next, 
for any positive compact operator $a \in K(H)^+$, 
let 
$\lambda(a) = (\lambda_1(a),\lambda_2(a),\dots, \lambda_n(a), \dots)$ be the 
eigenvalue sequence of $a$.   We have  
$a = \sum_{n=1}^{\infty} \lambda_n(a)p_n$, where each $p_n$ is a one dimensional spectral projection of $a$.  
Put $b_n = \sum_{k=1}^{n} \lambda_k(a)p_k$.  Then the sequence $(b_n)_n$ is increasing and converges 
to $a$ in norm.  Since $\psi$ and $\psi_{\alpha}$  are lower semi-continuous on $K(H)^+$ in norm, monotone maps  and 
$\psi(b_n) = \psi_{\alpha}(b_n)$, we conclude that  $ \psi(a) = \psi_{\alpha}(a)$ by taking their limits. 
Therefore $\psi$ is a non-linear trace $\psi_{\alpha}$  of Sugeno type associated with  
a monotone increasing function $\alpha$. 

(3)$\Rightarrow$(2)  It is clear from the fact that any monotone increasing functions $f$ and $g$  in $C(\sigma(a))$ 
are comonotonic. 

(1)$\Rightarrow$(3) 
For $a \in K(H)^+$, 
let $a = \sum_{i=1}^{\infty} \lambda_i(a) p_i$
be the spectral decomposition of $a$, where 
$\lambda(a) = (\lambda_1(a),\lambda_2(a),\dots, \lambda_n(a),\dots )$ is the list of the 
eigenvalue sequence of $a$.  
For any comonotonic functions $f$ and $g$  in $C(\sigma(a))^+$ with $f(0) = g(0) = 0$, 
there exists an injective map 
$\tau: {\mathbb N} \rightarrow {\mathbb N}$ such that 
$$f(\lambda_{\tau(1)}(a)) \geq f(\lambda_{\tau(2)}(a)) \geq \cdots \geq f(\lambda_{\tau(n)}(a)) \geq \cdots,
$$
$$
g(\lambda_{\tau(1)}(a)) \geq g(\lambda_{\tau(2)}(a)) \geq \cdots \geq g(\lambda_{\tau(n)}(a)) \geq \cdots,
$$
and 
$f(a) = \sum_{i=1}^{\infty} f(\lambda_{\tau(i)}(a)) p_{\tau(i)}$, 
               $g(a) = \sum_{i=1}^{\infty} g(\lambda_{\tau(i)}(a)) p_{\tau(i)}$.  Moreover 
$\lambda_{\tau(i)}(a)$ converges to $0$. 
Considering this fact, we can prove  (1)$\Rightarrow$(3) similarly as (1)$\Rightarrow$(2). 
\end{proof}

%%%
%%%   Section 6
%%%

\section{Triangle inequality for non-linear traces of Sugeno type}

In this section we discuss the triangle inequality and  trace class operators for non-linear traces of Sugeno type. 
We always assume that 
$\alpha: {\mathbb N}_0 \rightarrow [0, \infty)$ is  a monotone increasing function with $\alpha(0) = 0$ 
and also $\alpha(1) > 0$. 

%%%%%%%%%%%%%%%%%%%%%%%%%%%%
%  Lemma 6.1
%%%%%%%%%%%%%%%%%%%%%%%%%%%% 
\begin{lemma}
For $a \in K(H)^+$ and projections $p,q \in B(H)$, if $0 \leq p \leq q$, 
then $\psi_{\alpha}(pap) \leq \psi_{\alpha}(qaq)$. 
\end{lemma}
\begin{proof}
Siince $\lambda_{n}(pap) = \lambda_{n}(pqaqp) \leq \|p\|^2\lambda_{n}(qaq) \leq \lambda_{n}(qaq)$, 
$$
\psi_{\alpha}(pap) = \lor_{i=1} ^{\infty} ( \lambda_i(pap) \land \alpha(i)) 
\leq \lor_{i=1} ^{\infty} ( \lambda_i(qaq) \land \alpha(i)) = \psi_{\alpha}(qaq).
$$
\end{proof}

We need the following observation: 

%%%%%%%%%%%%%%%%%%%%%%%%%%%%%
%    Lemma 6.2
%%%%%%%%%%%%%%%%%%%%%%%%%%%%%

\begin{lemma}
\label{prop:observation}
Let $\alpha: {\mathbb N}_0 \rightarrow [0, \infty)$ be a monotone increasing function with $\alpha(0) = 0$ 
and $\alpha(1) > 0$. 
For  $a \in K(H)^+$, 
let $a = \sum_{i=1}^{\infty} \lambda_i(a) p_i$
be the spectral decomposition of $a$, where 
$$
\lambda(a) = (\lambda_1(a),\lambda_2(a),\dots, \lambda_n(a),\dots )
$$ 
is the list of the 
eigenvalues of $a$ in decreasing order with counting multiplicities. 
Then we have the following:
\begin{enumerate}
  \item[$(1)$] There exists a finite rank projection $p$ such that
$$ pap \ge \psi_\alpha(a)p  \text{ and } \psi_\alpha(a) \le \alpha(\dim p) . $$  
  \item[$(2)$] There exists a finite rank projection $q$ such that
$$ (I-q)a(I-q) \le \psi_\alpha(a)(I-q)  \text{ and } \psi_\alpha(a) \ge \alpha(\dim q) . $$
\end{enumerate}
\end{lemma}
\begin{proof}
Suppose that  $a \in K(H)^+$ satisfies that $\lambda_1(a) \geq \alpha(1)$.  
Choose a unique natural number $n \geq 2$ such that $\lambda_{n-1}(a) \geq \alpha(n-1)$ and 
$\lambda_{n}(a) < \alpha(n)$. Then $\psi_{\alpha}(a) = \lambda_{n}(a) \lor\alpha(n-1)$. 
%Moreover we have the following: \\

(a) Consider the case that $\lambda_{n}(a) \geq \alpha(n-1)$. 
Then  $\psi_{\alpha}(a) = \lambda_{n}(a)$.
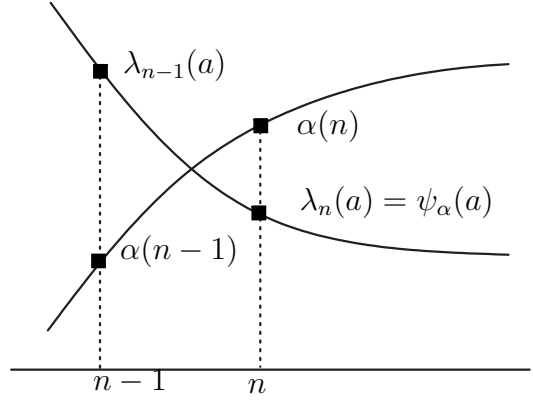
\begin{center}
%WinTpicVersion4.32a
{\unitlength 0.1in%
\begin{picture}(26.8500,19.1500)(3.1000,-25.0500)%
% LINE 1 0 3 0 Black White  
% 2 310 2505 2995 2505
% 
\special{pn 13}%
\special{pa 310 2505}%
\special{pa 2995 2505}%
\special{fp}%
% SPLINE 1 0 3 0 Black White  
% 4 500 2300 1340 1385 2805 915 2885 910
% 
\special{pn 13}%
\special{pa 500 2300}%
\special{pa 520 2274}%
\special{pa 539 2248}%
\special{pa 559 2222}%
\special{pa 578 2195}%
\special{pa 618 2143}%
\special{pa 637 2117}%
\special{pa 657 2092}%
\special{pa 717 2014}%
\special{pa 757 1964}%
\special{pa 778 1938}%
\special{pa 818 1888}%
\special{pa 839 1864}%
\special{pa 860 1839}%
\special{pa 881 1815}%
\special{pa 902 1790}%
\special{pa 923 1766}%
\special{pa 945 1743}%
\special{pa 967 1719}%
\special{pa 988 1696}%
\special{pa 1010 1673}%
\special{pa 1033 1650}%
\special{pa 1055 1627}%
\special{pa 1124 1561}%
\special{pa 1147 1540}%
\special{pa 1195 1498}%
\special{pa 1219 1478}%
\special{pa 1269 1438}%
\special{pa 1319 1400}%
\special{pa 1397 1346}%
\special{pa 1478 1295}%
\special{pa 1506 1279}%
\special{pa 1533 1263}%
\special{pa 1561 1248}%
\special{pa 1590 1233}%
\special{pa 1618 1218}%
\special{pa 1676 1190}%
\special{pa 1705 1177}%
\special{pa 1735 1163}%
\special{pa 1765 1151}%
\special{pa 1795 1138}%
\special{pa 1855 1114}%
\special{pa 1885 1103}%
\special{pa 1947 1081}%
\special{pa 2009 1061}%
\special{pa 2040 1052}%
\special{pa 2071 1042}%
\special{pa 2103 1033}%
\special{pa 2134 1025}%
\special{pa 2166 1016}%
\special{pa 2198 1008}%
\special{pa 2230 1001}%
\special{pa 2261 993}%
\special{pa 2293 986}%
\special{pa 2326 980}%
\special{pa 2358 973}%
\special{pa 2422 961}%
\special{pa 2486 951}%
\special{pa 2519 946}%
\special{pa 2551 941}%
\special{pa 2583 937}%
\special{pa 2616 933}%
\special{pa 2648 929}%
\special{pa 2776 917}%
\special{pa 2872 911}%
\special{pa 2885 910}%
\special{fp}%
% SPLINE 1 0 3 0 Black White  
% 3 515 590 1600 1695 2885 1905
% 
\special{pn 13}%
\special{pa 515 590}%
\special{pa 535 618}%
\special{pa 555 645}%
\special{pa 575 673}%
\special{pa 595 700}%
\special{pa 615 728}%
\special{pa 635 755}%
\special{pa 655 783}%
\special{pa 715 864}%
\special{pa 736 891}%
\special{pa 776 945}%
\special{pa 797 971}%
\special{pa 818 998}%
\special{pa 838 1024}%
\special{pa 859 1050}%
\special{pa 880 1075}%
\special{pa 901 1101}%
\special{pa 943 1151}%
\special{pa 965 1176}%
\special{pa 986 1201}%
\special{pa 1030 1249}%
\special{pa 1052 1272}%
\special{pa 1074 1296}%
\special{pa 1096 1319}%
\special{pa 1119 1341}%
\special{pa 1142 1364}%
\special{pa 1165 1386}%
\special{pa 1234 1449}%
\special{pa 1306 1509}%
\special{pa 1331 1528}%
\special{pa 1355 1546}%
\special{pa 1405 1582}%
\special{pa 1431 1599}%
\special{pa 1457 1615}%
\special{pa 1482 1631}%
\special{pa 1509 1647}%
\special{pa 1535 1662}%
\special{pa 1589 1690}%
\special{pa 1617 1703}%
\special{pa 1644 1715}%
\special{pa 1672 1727}%
\special{pa 1701 1739}%
\special{pa 1729 1750}%
\special{pa 1787 1770}%
\special{pa 1817 1780}%
\special{pa 1846 1788}%
\special{pa 1876 1797}%
\special{pa 1906 1805}%
\special{pa 1937 1812}%
\special{pa 1967 1819}%
\special{pa 1998 1826}%
\special{pa 2060 1838}%
\special{pa 2092 1844}%
\special{pa 2123 1849}%
\special{pa 2155 1854}%
\special{pa 2187 1858}%
\special{pa 2219 1863}%
\special{pa 2251 1867}%
\special{pa 2284 1870}%
\special{pa 2316 1874}%
\special{pa 2349 1877}%
\special{pa 2381 1880}%
\special{pa 2414 1882}%
\special{pa 2447 1885}%
\special{pa 2480 1887}%
\special{pa 2514 1889}%
\special{pa 2613 1895}%
\special{pa 2647 1896}%
\special{pa 2680 1898}%
\special{pa 2748 1900}%
\special{pa 2781 1901}%
\special{pa 2815 1903}%
\special{pa 2849 1904}%
\special{pa 2882 1905}%
\special{pa 2885 1905}%
\special{fp}%
% LINE 1 2 3 0 Black White  
% 2 770 930 770 2500
% 
\special{pn 13}%
\special{pa 770 930}%
\special{pa 770 2500}%
\special{dt 0.045}%
% LINE 1 2 3 0 Black White  
% 2 1600 1215 1600 2490
% 
\special{pn 13}%
\special{pa 1600 1215}%
\special{pa 1600 2490}%
\special{dt 0.045}%
% BOX 1 2 0 0 Black Black  
% 2 730 1905 795 1970
% 
\special{pn 0}%
\special{sh 1.000}%
\special{pa 730 1905}%
\special{pa 795 1905}%
\special{pa 795 1970}%
\special{pa 730 1970}%
\special{pa 730 1905}%
\special{ip}%
\special{pn 13}%
\special{pa 730 1905}%
\special{pa 795 1905}%
\special{pa 795 1970}%
\special{pa 730 1970}%
\special{pa 730 1905}%
\special{fp}%
%\special{dt 0.045}%
% BOX 1 0 0 0 Black Black  
% 2 740 915 805 980
% 
\special{pn 0}%
\special{sh 1.000}%
\special{pa 740 915}%
\special{pa 805 915}%
\special{pa 805 980}%
\special{pa 740 980}%
\special{pa 740 915}%
\special{ip}%
\special{pn 13}%
\special{pa 740 915}%
\special{pa 805 915}%
\special{pa 805 980}%
\special{pa 740 980}%
\special{pa 740 915}%
\special{pa 805 915}%
\special{fp}%
% BOX 1 0 0 0 Black White  
% 2 1570 1200 1635 1265
% 
\special{pn 0}%
\special{sh 1.000}%
\special{pa 1570 1200}%
\special{pa 1635 1200}%
\special{pa 1635 1265}%
\special{pa 1570 1265}%
\special{pa 1570 1200}%
\special{ip}%
\special{pn 13}%
\special{pa 1570 1200}%
\special{pa 1635 1200}%
\special{pa 1635 1265}%
\special{pa 1570 1265}%
\special{pa 1570 1200}%
\special{pa 1635 1200}%
\special{fp}%
% STR 2 0 3 0 Black White  
% 4 735 2580 735 2630 2 0 0 0
% n-1
\put(7.3500,-26.2500){\makebox(0,0)[lb]{$n-1$}}%
% STR 2 0 3 0 Black White  
% 4 1535 2575 1535 2625 2 0 0 0
% n
\put(15.3500,-26.2500){\makebox(0,0)[lb]{$n$}}%
% STR 2 0 3 0 Black White  
% 4 875 1920 875 1970 2 0 0 0
% a(n-1)
\put(8.7500,-19.7000){\makebox(0,0)[lb]{$\alpha(n-1)$}}%
% STR 2 0 3 0 Black White  
% 4 890 950 890 1000 2 0 0 0
% ln-1(a)
\put(8.9000,-10.0000){\makebox(0,0)[lb]{$\lambda_{n-1}(a)$}}%
% STR 2 0 3 0 Black White  
% 4 1790 1265 1790 1315 2 0 0 0
% a(n)
\put(17.9000,-13.1500){\makebox(0,0)[lb]{$\alpha(n)$}}%
% STR 2 0 3 0 Black White  
% 4 1805 1665 1805 1715 2 0 0 0
% ln(a)=pa(a)
\put(18.0500,-17.1500){\makebox(0,0)[lb]{$\lambda_n(a)=\psi_{\alpha}(a)$}}%
% BOX 1 0 0 0 Black White  
% 2 1560 1655 1625 1720
% 
\special{pn 0}%
\special{sh 1.000}%
\special{pa 1560 1655}%
\special{pa 1625 1655}%
\special{pa 1625 1720}%
\special{pa 1560 1720}%
\special{pa 1560 1655}%
\special{ip}%
\special{pn 13}%
\special{pa 1560 1655}%
\special{pa 1625 1655}%
\special{pa 1625 1720}%
\special{pa 1560 1720}%
\special{pa 1560 1655}%
\special{pa 1625 1655}%
\special{fp}%
\end{picture}}%

\end{center} 
Put $p = \sum_{k=1}^n p_k$. We have that 
$$\psi_{\alpha}(a) = \psi_{\alpha}(pap),\ \  
pap \geq \lambda_{n}(a) p = \psi_{\alpha}(a)p  \ \text{ and }
\psi_{\alpha}(a) < \alpha (\dim p). 
$$
For any projection $q$ with $p \leq q$ we have that 
$$
\psi_{\alpha}(a) =\psi_{\alpha}(pap) = \psi_{\alpha}(qaq). 
$$ 
Put $q_0 = \sum_{k=1}^{n-1} p_k$. Then 
$$
\alpha (\dim q_0) = \alpha(n-1) \leq \lambda_{n}(a) = \psi_{\alpha}(a)     \text{ and }
$$
$$
(I-q_0)a (I-q_0) \leq \lambda_{n}(a)(I-q_0) = \psi_{\alpha}(a)(I-q_0). 
$$

(b) Consider the case that $\lambda_{n}(a) < \alpha(n-1)$. 
Then  $\psi_{\alpha}(a) = \alpha(n-1)$.
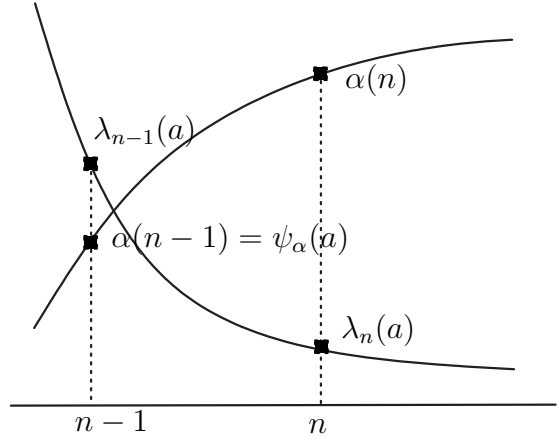
\begin{center}
%WinTpicVersion4.32a
{\unitlength 0.1in%
\begin{picture}(28.2000,21.0000)(2.9000,-25.0500)%
% LINE 1 0 3 0 Black White  
% 2 290 2505 3110 2500
% 
\special{pn 13}%
\special{pa 290 2505}%
\special{pa 3110 2500}%
\special{fp}%
% SPLINE 1 0 3 0 Black White  
% 4 410 2100 1010 1285 2795 600 2885 595
% 
\special{pn 13}%
\special{pa 410 2100}%
\special{pa 427 2072}%
\special{pa 444 2045}%
\special{pa 495 1961}%
\special{pa 512 1934}%
\special{pa 529 1906}%
\special{pa 547 1879}%
\special{pa 564 1852}%
\special{pa 581 1824}%
\special{pa 617 1770}%
\special{pa 635 1744}%
\special{pa 671 1690}%
\special{pa 689 1664}%
\special{pa 746 1586}%
\special{pa 765 1561}%
\special{pa 785 1536}%
\special{pa 804 1511}%
\special{pa 824 1486}%
\special{pa 845 1461}%
\special{pa 865 1437}%
\special{pa 886 1413}%
\special{pa 907 1390}%
\special{pa 929 1366}%
\special{pa 951 1344}%
\special{pa 973 1321}%
\special{pa 1042 1255}%
\special{pa 1066 1234}%
\special{pa 1090 1214}%
\special{pa 1114 1193}%
\special{pa 1139 1173}%
\special{pa 1164 1154}%
\special{pa 1190 1134}%
\special{pa 1216 1116}%
\special{pa 1242 1097}%
\special{pa 1268 1079}%
\special{pa 1295 1061}%
\special{pa 1349 1027}%
\special{pa 1405 993}%
\special{pa 1433 977}%
\special{pa 1462 962}%
\special{pa 1490 946}%
\special{pa 1519 931}%
\special{pa 1548 917}%
\special{pa 1578 903}%
\special{pa 1607 889}%
\special{pa 1637 875}%
\special{pa 1697 849}%
\special{pa 1728 836}%
\special{pa 1758 824}%
\special{pa 1789 812}%
\special{pa 1820 801}%
\special{pa 1851 789}%
\special{pa 1882 778}%
\special{pa 1913 768}%
\special{pa 1945 758}%
\special{pa 1976 748}%
\special{pa 2008 738}%
\special{pa 2104 711}%
\special{pa 2200 687}%
\special{pa 2233 680}%
\special{pa 2297 666}%
\special{pa 2330 660}%
\special{pa 2362 653}%
\special{pa 2395 648}%
\special{pa 2427 642}%
\special{pa 2460 637}%
\special{pa 2492 632}%
\special{pa 2525 627}%
\special{pa 2589 619}%
\special{pa 2622 615}%
\special{pa 2654 612}%
\special{pa 2687 609}%
\special{pa 2751 603}%
\special{pa 2879 595}%
\special{pa 2885 595}%
\special{fp}%
% SPLINE 1 0 3 0 Black White  
% 4 415 405 1150 1895 2895 2315 2895 2315
% 
\special{pn 13}%
\special{pa 415 405}%
\special{pa 435 473}%
\special{pa 445 506}%
\special{pa 455 540}%
\special{pa 466 574}%
\special{pa 476 607}%
\special{pa 486 641}%
\special{pa 496 674}%
\special{pa 507 708}%
\special{pa 517 741}%
\special{pa 539 807}%
\special{pa 549 840}%
\special{pa 571 906}%
\special{pa 582 938}%
\special{pa 594 970}%
\special{pa 605 1002}%
\special{pa 629 1066}%
\special{pa 665 1159}%
\special{pa 678 1190}%
\special{pa 730 1310}%
\special{pa 758 1368}%
\special{pa 773 1396}%
\special{pa 787 1425}%
\special{pa 802 1452}%
\special{pa 817 1480}%
\special{pa 833 1507}%
\special{pa 848 1533}%
\special{pa 865 1560}%
\special{pa 881 1585}%
\special{pa 898 1611}%
\special{pa 915 1636}%
\special{pa 933 1660}%
\special{pa 950 1684}%
\special{pa 969 1708}%
\special{pa 987 1731}%
\special{pa 1006 1753}%
\special{pa 1026 1776}%
\special{pa 1066 1818}%
\special{pa 1087 1839}%
\special{pa 1108 1858}%
\special{pa 1130 1878}%
\special{pa 1152 1896}%
\special{pa 1174 1915}%
\special{pa 1197 1932}%
\special{pa 1245 1966}%
\special{pa 1269 1982}%
\special{pa 1319 2012}%
\special{pa 1371 2040}%
\special{pa 1397 2053}%
\special{pa 1451 2079}%
\special{pa 1479 2091}%
\special{pa 1535 2113}%
\special{pa 1564 2124}%
\special{pa 1622 2144}%
\special{pa 1712 2171}%
\special{pa 1742 2179}%
\special{pa 1773 2187}%
\special{pa 1804 2194}%
\special{pa 1836 2201}%
\special{pa 1867 2208}%
\special{pa 1899 2215}%
\special{pa 1963 2227}%
\special{pa 1996 2233}%
\special{pa 2128 2253}%
\special{pa 2298 2273}%
\special{pa 2332 2276}%
\special{pa 2367 2279}%
\special{pa 2401 2282}%
\special{pa 2506 2291}%
\special{pa 2541 2293}%
\special{pa 2576 2296}%
\special{pa 2611 2298}%
\special{pa 2647 2301}%
\special{pa 2717 2305}%
\special{pa 2753 2307}%
\special{pa 2788 2309}%
\special{pa 2824 2311}%
\special{pa 2859 2313}%
\special{pa 2895 2315}%
\special{fp}%
% LINE 1 2 3 0 Black White  
% 2 705 1240 705 2495
% 
\special{pn 13}%
\special{pa 705 1240}%
\special{pa 705 2495}%
\special{dt 0.045}%
% LINE 1 2 3 0 Black White  
% 2 1895 765 1895 2495
% 
\special{pn 13}%
\special{pa 1895 765}%
\special{pa 1895 2495}%
\special{dt 0.045}%
% BOX 1 2 0 0 Black White  
% 2 670 1210 735 1275
% 
\special{pn 0}%
\special{sh 1.000}%
\special{pa 670 1210}%
\special{pa 735 1210}%
\special{pa 735 1275}%
\special{pa 670 1275}%
\special{pa 670 1210}%
\special{ip}%
\special{pn 13}%
\special{pa 670 1210}%
\special{pa 735 1210}%
\special{pa 735 1275}%
\special{pa 670 1275}%
\special{pa 670 1210}%
\special{dt 0.045}%
% BOX 1 2 0 0 Black White  
% 2 670 1620 735 1685
% 
\special{pn 0}%
\special{sh 1.000}%
\special{pa 670 1620}%
\special{pa 735 1620}%
\special{pa 735 1685}%
\special{pa 670 1685}%
\special{pa 670 1620}%
\special{ip}%
\special{pn 13}%
\special{pa 670 1620}%
\special{pa 735 1620}%
\special{pa 735 1685}%
\special{pa 670 1685}%
\special{pa 670 1620}%
\special{dt 0.045}%
% BOX 1 2 0 0 Black White  
% 2 1860 740 1925 805
% 
\special{pn 0}%
\special{sh 1.000}%
\special{pa 1860 740}%
\special{pa 1925 740}%
\special{pa 1925 805}%
\special{pa 1860 805}%
\special{pa 1860 740}%
\special{ip}%
\special{pn 13}%
\special{pa 1860 740}%
\special{pa 1925 740}%
\special{pa 1925 805}%
\special{pa 1860 805}%
\special{pa 1860 740}%
\special{dt 0.045}%
% BOX 1 2 0 0 Black White  
% 2 1865 2165 1930 2230
% 
\special{pn 0}%
\special{sh 1.000}%
\special{pa 1865 2165}%
\special{pa 1930 2165}%
\special{pa 1930 2230}%
\special{pa 1865 2230}%
\special{pa 1865 2165}%
\special{ip}%
\special{pn 13}%
\special{pa 1865 2165}%
\special{pa 1930 2165}%
\special{pa 1930 2230}%
\special{pa 1865 2230}%
\special{pa 1865 2165}%
\special{dt 0.045}%
% STR 2 0 3 0 Black White  
% 4 620 2580 620 2630 2 0 0 0
% n-1
\put(6.2000,-26.5000){\makebox(0,0)[lb]{$n-1$}}%
% STR 2 0 3 0 Black White  
% 4 1830 2600 1830 2650 2 0 0 0
% n
\put(18.3000,-26.5000){\makebox(0,0)[lb]{$n$}}%
% STR 2 0 3 0 Black White  
% 4 715 1105 715 1155 2 0 0 0
% ln-1(a)
\put(7.1500,-11.5500){\makebox(0,0)[lb]{$\lambda_{n-1}(a)$}}%
% STR 2 0 3 0 Black White  
% 4 2020 840 2020 890 2 0 0 0
% a(n)
\put(20.2000,-8.9000){\makebox(0,0)[lb]{$\alpha(n)$}}%
% STR 2 0 3 0 Black White  
% 4 2000 2145 2000 2195 2 0 0 0
% ln(a)
\put(20.0000,-21.9500){\makebox(0,0)[lb]{$\lambda_n(a)$}}%
% STR 2 0 3 0 Black White  
% 4 810 1670 810 1720 2 0 0 0
% a(n-1)=pa(a)
\put(8.1000,-17.2000){\makebox(0,0)[lb]{$\alpha(n-1)=\psi_{\alpha}(a)$}}%
\end{picture}}%

\end{center} 
Put $p = \sum_{k=1}^{n-1} p_k$. We have that 
$$\psi_{\alpha}(a) = \psi_{\alpha}(pap),\ \  
pap \geq \lambda_{n-1}(a) p  \geq \alpha(n-1)p = \psi_{\alpha}(a)p   \text{ and }
$$
$$
\psi_{\alpha}(a) = \alpha (\dim p). 
$$
For any projection $q$ with $p \leq q$ we have that 
$$
\psi_{\alpha}(a) =\psi_{\alpha}(pap) = \psi_{\alpha}(qaq). 
$$ 
Put $q_0 = \sum_{k=1}^{n-1} p_k$. Then 
$$
\alpha (\dim q_0) = \alpha(n-1)  = \psi_{\alpha}(a)    \text{ and }
$$
$$
(I-q_0)a (I-q_0) \leq \lambda_{n}(a)(I-q_0) \leq \alpha(n-1)(I-q_0) = \psi_{\alpha}(a)(I-q_0). 
$$

(c)  Suppose that  $a \in K(H)^+$ satisfies that $\lambda_1(a) < \alpha(1)$. 
This is the case that $n = 1$ and $\lambda_1(a) \geq \alpha(1-1) = 0$. 
Then  $\psi_{\alpha}(a) = \lambda_{1}(a)$.
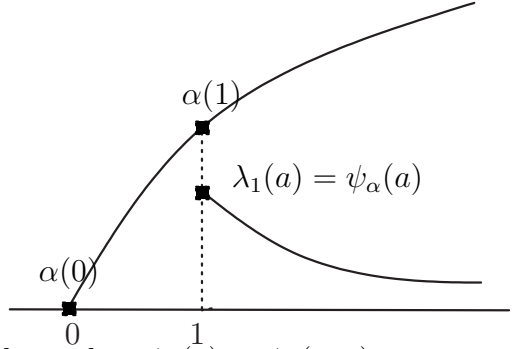
\begin{center}
%WinTpicVersion4.32a
{\unitlength 0.1in%
\begin{picture}(25.8500,16.4000)(2.1500,-24.4000)%
% LINE 1 0 3 0 Black White  
% 2 215 2400 2800 2405
% 
\special{pn 13}%
\special{pa 215 2400}%
\special{pa 2800 2405}%
\special{fp}%
% SPLINE 1 0 3 0 Black White  
% 4 515 2395 1200 1460 2620 800 2620 800
% 
\special{pn 13}%
\special{pa 515 2395}%
\special{pa 532 2367}%
\special{pa 549 2338}%
\special{pa 565 2310}%
\special{pa 582 2281}%
\special{pa 684 2113}%
\special{pa 702 2085}%
\special{pa 719 2057}%
\special{pa 737 2030}%
\special{pa 754 2002}%
\special{pa 808 1921}%
\special{pa 826 1895}%
\special{pa 845 1868}%
\special{pa 863 1842}%
\special{pa 882 1816}%
\special{pa 901 1791}%
\special{pa 920 1765}%
\special{pa 960 1715}%
\special{pa 1020 1643}%
\special{pa 1041 1619}%
\special{pa 1083 1573}%
\special{pa 1105 1551}%
\special{pa 1127 1528}%
\special{pa 1149 1507}%
\special{pa 1172 1485}%
\special{pa 1195 1465}%
\special{pa 1218 1444}%
\special{pa 1266 1404}%
\special{pa 1290 1385}%
\special{pa 1315 1366}%
\special{pa 1365 1330}%
\special{pa 1391 1312}%
\special{pa 1443 1278}%
\special{pa 1469 1262}%
\special{pa 1496 1245}%
\special{pa 1523 1229}%
\special{pa 1550 1214}%
\special{pa 1634 1169}%
\special{pa 1662 1155}%
\special{pa 1691 1140}%
\special{pa 1719 1127}%
\special{pa 1748 1113}%
\special{pa 1777 1100}%
\special{pa 1807 1087}%
\special{pa 1836 1074}%
\special{pa 1866 1061}%
\special{pa 1896 1049}%
\special{pa 1926 1036}%
\special{pa 1956 1024}%
\special{pa 1987 1013}%
\special{pa 2017 1001}%
\special{pa 2048 989}%
\special{pa 2078 978}%
\special{pa 2140 956}%
\special{pa 2172 945}%
\special{pa 2203 934}%
\special{pa 2234 924}%
\special{pa 2266 913}%
\special{pa 2297 903}%
\special{pa 2329 892}%
\special{pa 2360 882}%
\special{pa 2424 862}%
\special{pa 2456 851}%
\special{pa 2487 841}%
\special{pa 2583 811}%
\special{pa 2615 802}%
\special{pa 2620 800}%
\special{fp}%
% SPLINE 1 0 3 0 Black White  
% 5 1210 1785 1665 2095 2610 2260 2655 2260 2655 2260
% 
\special{pn 13}%
\special{pa 1210 1785}%
\special{pa 1235 1805}%
\special{pa 1261 1825}%
\special{pa 1286 1845}%
\special{pa 1312 1865}%
\special{pa 1337 1885}%
\special{pa 1363 1904}%
\special{pa 1389 1924}%
\special{pa 1415 1942}%
\special{pa 1441 1961}%
\special{pa 1468 1979}%
\special{pa 1494 1997}%
\special{pa 1548 2031}%
\special{pa 1575 2047}%
\special{pa 1603 2063}%
\special{pa 1631 2078}%
\special{pa 1659 2092}%
\special{pa 1688 2106}%
\special{pa 1717 2119}%
\special{pa 1746 2131}%
\special{pa 1806 2153}%
\special{pa 1836 2163}%
\special{pa 1866 2172}%
\special{pa 1928 2190}%
\special{pa 1959 2197}%
\special{pa 1991 2204}%
\special{pa 2022 2211}%
\special{pa 2086 2223}%
\special{pa 2150 2233}%
\special{pa 2182 2237}%
\special{pa 2215 2241}%
\special{pa 2247 2244}%
\special{pa 2280 2247}%
\special{pa 2312 2250}%
\special{pa 2345 2252}%
\special{pa 2377 2254}%
\special{pa 2410 2255}%
\special{pa 2442 2257}%
\special{pa 2475 2258}%
\special{pa 2507 2259}%
\special{pa 2540 2259}%
\special{pa 2572 2260}%
\special{pa 2655 2260}%
\special{fp}%
% LINE 1 2 3 0 Black White  
% 4 1260 2395 1210 2405 1205 1460 1225 1475
% 
\special{pn 13}%
\special{pa 1260 2395}%
\special{pa 1210 2405}%
\special{dt 0.045}%
\special{pa 1205 1460}%
\special{pa 1225 1475}%
\special{dt 0.045}%
% LINE 1 2 3 0 Black White  
% 2 1210 2400 1210 1455
% 
\special{pn 13}%
\special{pa 1210 2400}%
\special{pa 1210 1455}%
\special{dt 0.045}%
% STR 2 0 3 0 Black White  
% 4 500 2530 500 2580 2 0 0 0
% 0
\put(5.0000,-25.8000){\makebox(0,0)[lb]{$0$}}%
% STR 2 0 3 0 Black White  
% 4 1145 2530 1145 2580 2 0 0 0
% 1
\put(11.4500,-25.8000){\makebox(0,0)[lb]{$1$}}%
% BOX 1 2 0 0 Black White  
% 2 1180 1420 1245 1485
% 
\special{pn 0}%
\special{sh 1.000}%
\special{pa 1180 1420}%
\special{pa 1245 1420}%
\special{pa 1245 1485}%
\special{pa 1180 1485}%
\special{pa 1180 1420}%q
\special{ip}%
\special{pn 13}%
\special{pa 1180 1420}%
\special{pa 1245 1420}%
\special{pa 1245 1485}%
\special{pa 1180 1485}%
\special{pa 1180 1420}%
\special{dt 0.045}%
% BOX 1 2 0 0 Black White  
% 2 1180 1760 1245 1825
% 
\special{pn 0}%
\special{sh 1.000}%
\special{pa 1180 1760}%
\special{pa 1245 1760}%
\special{pa 1245 1825}%
\special{pa 1180 1825}%
\special{pa 1180 1760}%
\special{ip}%
\special{pn 13}%
\special{pa 1180 1760}%
\special{pa 1245 1760}%
\special{pa 1245 1825}%
\special{pa 1180 1825}%
\special{pa 1180 1760}%
\special{dt 0.045}%
% BOX 1 2 0 0 Black White  
% 2 490 2365 555 2430
% 
\special{pn 0}%
\special{sh 1.000}%
\special{pa 490 2365}%
\special{pa 555 2365}%
\special{pa 555 2430}%
\special{pa 490 2430}%
\special{pa 490 2365}%
\special{ip}%
\special{pn 13}%
\special{pa 490 2365}%
\special{pa 555 2365}%
\special{pa 555 2430}%
\special{pa 490 2430}%
\special{pa 490 2365}%
\special{dt 0.045}%
% STR 2 0 3 0 Black White  
% 4 365 2245 365 2295 2 0 0 0
% a(0)
\put(3.6500,-22.9500){\makebox(0,0)[lb]{$\alpha(0)$}}%
% STR 2 0 3 0 Black White  
% 4 1105 1320 1105 1370 2 0 0 0
% a(1)
\put(11.0500,-13.7000){\makebox(0,0)[lb]{$\alpha(1)$}}%
% STR 2 0 3 0 Black White  
% 4 1365 1755 1365 1805 2 0 0 0
% l1(1)=pa(a)
\put(13.6500,-18.0500){\makebox(0,0)[lb]{$\lambda_1(a)=\psi_{\alpha}(a)$}}%
\end{picture}}%

\end{center} 
Put $p = p_1$. We have that $\psi_{\alpha}(a) = \psi_{\alpha}(pap)$,  
$$ pap  = \lambda_{1}(a)p =  \psi_{\alpha}(a)p  \ \text{ and }
\psi_{\alpha}(a) < \alpha(1) = \alpha (\dim p). 
$$
For any projection $q$ with $p \leq q$ we have that 
$$
\psi_{\alpha}(a) =\psi_{\alpha}(pap) = \psi_{\alpha}(qaq). 
$$ 
Put $q_0 = 0$. Then 
\begin{gather*}
\alpha (\dim q_0) = \alpha(0) \leq \lambda_{1}(a) = \psi_{\alpha}(a)    \ \text{ and }  \\
(I-q_0)a (I-q_0) \leq \lambda_{1}(a)(I-q_0) = \psi_{\alpha}(a)(I-q_0). 
\end{gather*}
\end{proof}

%\begin{proof} A carefull observation implies the statement in each case. 
%\end{proof}

The following Theorem is an important expression of the value of the non-linear trace of Sugeno type 
and is a key to prove the triangle inequality. 

%%%%%%%%%%%%%%%%%%%%%%%%%%%%%%%%
%  Theorem 6.3
%%%%%%%%%%%%%%%%%%%%%%%%%%%%%%%%
\begin{theorem}
\label{prop:max}
Let $\alpha: {\mathbb N}_0 \rightarrow [0, \infty)$ be a monotone increasing function with $\alpha(0) = 0$ 
and $\alpha(1) > 0$. 
For  $a \in K(H)^+$, we have that
\begin{align*}
\psi_{\alpha}(a) = \max \{ \lambda \geq 0 \ | \ & \exists p \text{ a finite rank projection s.t. } \\
  & \qquad \qquad \lambda \leq \alpha (\dim p), \ pap \geq \lambda p \}. 
\end{align*}
\end{theorem}
\begin{proof}
Choose $p$ in the Lemma \ref{prop:observation}(1), we have that 
\begin{align*}
\psi_{\alpha}(a) \leq \max \{ \lambda \geq 0 \ | \ & \exists p \text{ a finite rank projection s.t. } \\ 
       & \qquad \qquad \lambda \leq \alpha (\dim p), \ pap \geq \lambda p \}. 
\end{align*}

To show the converse inequality, take any finite rank projection $p$ such that 
$$
\lambda \leq \alpha (\dim p)  \text{ and } \ pap \geq \lambda p.
$$
Then we have that 
$$
\psi_{\alpha}(a) \geq \psi_{\alpha}(pap) \geq \psi_{\alpha}(\lambda p) . 
$$
Since $\lambda \leq \alpha (\dim p)$, 
$$
\psi_{\alpha}(\lambda p)  = \lambda \land \alpha (\dim p) = \lambda.
$$
This implies the conclusion. 

\end{proof}

We shall show that  if $\alpha$ is concave, then the triangle inequality holds. 

%%%%%%%%%%%%%%%%%%%%%%%%%%%%%
%     Theorem 6.4
%%%%%%%%%%%%%%%%%%%%%%%%%%%%%
\begin{theorem} 
Let $\psi = \psi_{\alpha}$ be  a non-linear trace of Sugeno type associated with  
a monotone increasing function $\alpha: {\mathbb N}_0  \rightarrow [0, \infty)$  with $\alpha(0) = 0$ and 
$\alpha(1) > 0$.   
Define $||a||_{\alpha}:= \psi_{\alpha}(|a|)$ for $a \in K(H)$. 
Assume that $\alpha$ is concave in the sense that 
$\frac{\alpha(i +1) + \alpha(i - 1)}{2} \leq \alpha(i), \;  (i =1,2,3, $ $\ldots)$.
Then 
$|| \ ||_{\alpha}$ satisfies the triangle inequality: for any $b,c \in K(H) $, 
$$
||b + c||_{\alpha} \leq ||b ||_{\alpha} + || c||_{\alpha}.
$$
\end{theorem}
\begin{proof}
First we assume that $b,c \in K(H)$ are positive.  Put $a = b + c$.  
Let $a = \sum_{i=1}^{\infty} \lambda_i(a) p_i$
be the spectral decomposition of $a$, where 
$$
\lambda(a) = (\lambda_1(a),\lambda_2(a),\dots, \lambda_n(a),\dots )
$$ 
is the list of the 
eigenvalues of $a$ in decreasing order with counting multiplicities. 
Let $\lambda := \psi_{\alpha}(a)$.  
By Lemma \ref{prop:observation}(1), there exist a natural number $n$ and a 
projection $p = \sum_{i=1}^{n} p_i$ such that 
$$
\dim p = n, \ \ pap \geq \lambda p, \ \ \lambda  \leq \alpha(n) = \alpha(\dim p) \text{ and } 
\psi_{\alpha}(a) = \psi_{\alpha}(pap)
$$
To get $||b + c||_{\alpha} \leq ||b ||_{\alpha} + || c||_{\alpha}$, it is enough to show that 
$$
\psi_{\alpha}(pbp + pcp) \leq \psi_{\alpha}(pbp)  + \psi_{\alpha}(pcp). 
$$
In fact, if this is shown, then 
$$
||b + c||_{\alpha} =  \psi_{\alpha}(a) = \psi_{\alpha}(pap) \leq \psi_{\alpha}(pbp)  + \psi_{\alpha}(pcp) 
\leq \psi_{\alpha}(b) + \psi_{\alpha}(c) 
$$

Let $\mu := \psi_{\alpha}(pbp)$. Apply Lemma \ref{prop:observation}(2) to $pbp$. Then there exists a projection 
$q_0 \leq p$ such that
$$
\alpha (\dim q_0) \leq  \psi_{\alpha}(pbp)    \ \text{ and }
$$
$$
(I-q_0)pbp (I-q_0) \leq  \psi_{\alpha}(pbp)(I-q_0). 
$$
Hence $(p-q_0)pbp (p-q_0) \leq  \psi_{\alpha}(pbp)(p-q_0) = \mu (p-q_0) . $  Since $\alpha$ is concave, 
$$
\alpha (\dim p)  =  \alpha (\dim (p-q_0) + \dim q_0) \leq  \alpha (\dim (p -q_0) ) +  \alpha (\dim q_0), 
$$
we have that 
$$
\lambda - \mu =  \psi_{\alpha}(a) -  \psi_{\alpha}(pbp) \leq  \alpha(\dim p) - \alpha (\dim q_0) 
\leq \alpha (\dim (p -q_0) ). 
$$
Then 
\begin{align*}
&  
(p-q_0)pcp (p-q_0) = (p-q_0)pap (p-q_0) - (p-q_0)pbp (p-q_0) \\
& \geq \lambda (p-q_0) - \mu(p-q_0) = (\lambda - \mu)(p-q_0)
\end{align*}

By Theorem \ref{prop:max},  we have that $\lambda - \mu \leq  \psi_{\alpha}(pcp)$. 
This means that 
$$
\psi_{\alpha}(pbp + pcp) = \psi_{\alpha}(pap) \leq \psi_{\alpha}(pbp)  + \psi_{\alpha}(pcp). 
$$
Next, we consider the general case such that $b,c \in K(H)$ are not necessarily positive. 
By a theorem of Akemann-Anderson-Pedersen \cite{A-A-P}, which  generalizes a Thompson's Theorem \cite{To}, 
there exist isometries $u,v \in B(H)$ such that 
$$
| b + c | \leq u|b|u^* + v|c|v^*.
$$
Terefore we have that 
\begin{align*}
 \psi_{\alpha}(| b + c |) & \leq \psi_{\alpha}(u|b|u^* + v|c|v^*) \\
   & \leq \psi_{\alpha}(u|b|u^*) + \psi_{\alpha}(v|c|v^*) =  \psi_{\alpha}(|b|) + \psi_{\alpha}(|c|).
\end{align*}
This implies the desired triangle inequality.
\end{proof}

%%%%%%%%%%%%%%%%%%%%%%%%%%%%%%%%%
%    Definition 6.5
%%%%%%%%%%%%%%%%%%%%%%%%%%%%%%%%%
\begin{definition} \rm
 Let $\psi= \psi_{\alpha}$ be  a non-linear trace of Sugeno type associated with  
a monotone increasing function $\alpha: {\mathbb N}_0  \rightarrow [0, \infty)$  with $\alpha(0) = 0$ and 
$\alpha(1) > 0$.  
Since for $a \in K(H)$ we have $||a||_{\alpha}:= \psi_{\alpha}(|a|) < \infty$, 
{\it any} $a \in K(H)$ is a {\it trace class operator for a non-linear trace}  $\psi_{\alpha}$  {\it of Sugeno type}.  
Assume that  $\alpha$ is concave.
Although $||\lambda a||_{\alpha} = |\lambda| ||a||_{\alpha}$ does not hold in general,  
the triangle inequality above implies that 
$K(H)$ with $d(a,b) := ||a -b||_{\alpha}$ is a metric space. 
Since it holds the properties
\begin{itemize}
\item $\|a\|<\alpha(1)$ implies $\|a\|=\|a\|_\alpha$
\item $\|a\|_\alpha<\alpha(1)$ implies $\|a\|=\|a\|_\alpha$,
\end{itemize}
this metric space is isomorphic to the Banach space $(K(H), \|\cdot\|)$ as a topological space.

We shall extend $\psi_{\alpha}$ to $K(H)$ as follows: Let $a \in K(H)$. Consider the decomposition 
$$
a = \frac{1}{2}(a + a^*) + i  \ \frac{1}{2i}(a - a^*) = a_1 -a_2 + i(a_3 -a_4)
$$
with $a_1, a_2, a_3, a_4 \in (K(H))^+$ and $a_1a_2 = a_3a_4 = 0$.  
We define a non-linear trace 
$\psi_{\alpha} : K(H) \rightarrow  {\mathbb C}$ by 
$$
\psi_{\alpha}(a) := \psi_{\alpha}(a_1) - \psi_{\alpha}(a_2) + i( \psi_{\alpha}(a_3) - \psi_{\alpha}(a_4)),
$$
where we use the same symbol $\psi_{\alpha}$ for the extended non-linear trace. 
\end{definition}

\end{document}